\documentclass[twoside]{article}
\usepackage{amsmath,amssymb,amsthm,mathptmx,hyperref}
\usepackage{pstricks,pst-plot}
\definecolor{matheonblue}{RGB}{0,0,50}
\definecolor{matheonlightblue}{RGB}{139,0,0}
\definecolor{diffxblue}{RGB}{0,0,50}
\definecolor{diffxred}{RGB}{139,0,0}
\definecolor{light-gray}{RGB}{200,215,225}
\numberwithin{equation}{section}
\theoremstyle{remark}
\newtheorem{remark}{Remark}[section]
\theoremstyle{definition}
\newtheorem{definition}{Definition}[section]

\newtheorem{proposition}{Proposition}[section]
\newtheorem{lemma}{Lemma}[section]
\newtheorem{corollary}{Corollary}[section]
\newtheorem{theorem}{Theorem}[section]

\newcommand{\f}[1]{\pmb{#1}}
\DeclareMathOperator{\N}{\mathbb{N}}
\DeclareMathOperator{\R}{\mathbb{R}}

\DeclareMathOperator{\C}{\mathcal{C}}
\DeclareMathOperator{\F}{\mathcal{F}}
\DeclareMathOperator{\AC}{\mathcal{AC}}

\DeclareMathOperator{\V}{\f H^1_{0,\sigma}}
\DeclareMathOperator{\Vd}{(\f H^{1}_{0,\sigma})^*}

\DeclareMathOperator{\Ha}{\f L^2_{\sigma}}
\DeclareMathOperator{\Hb}{\f{H}^1_0}
\DeclareMathOperator{\Hc}{\f H^2 \cap \f H^1_0}
\DeclareMathOperator{\Hd}{\f H^{-1}}
\DeclareMathOperator{\He}{\f{H}^1}
\DeclareMathOperator{\Le}{\f{L}^2}
\DeclareMathOperator{\ra}{\rightarrow}
\newcommand{\de}{\mathrm{d}\/}
\DeclareMathOperator{\tr}{tr}
\DeclareMathOperator{\spa}{span}
\newcommand{\drei}{\text{\,\,\multiput(0,-2)(0,2){3}{$\cdot$}\enskip }}
\newcommand{\br}[1]{\frac{\de #1}{\de t}}
\newcommand{\pat}[2]{\frac{\partial #1}{\partial #2}}
\DeclareMathOperator{\di}{\nabla \cdot}

\DeclareMathOperator{\sym}{{sym}}
\DeclareMathOperator{\skw}{skw}
\newcommand{\sy}[1]{(\nabla \f {#1})_{{\sym}}}
\newcommand{\sk}[1]{(\nabla \f {#1})_{\skw}}
\renewcommand{\t}{\partial_t  }
\newcommand{\syn}[1]{(\nabla \f {#1}_n)_{\sym}}
\newcommand{\skn}[1]{(\nabla \f {#1}_n)_{\skw}}
\DeclareMathOperator{\curl}{\nabla \times }

\newcommand{\inte}[1]{\int_{\Omega}\left({ #1}\right) \de \f x}

\newcommand{\intte}[1]{\int_{0}^T{ #1} \de t}
\newcommand{\inttet}[1]{\int_{0}^{t}{ #1} \de s}
\newcommand{\intet}[1]{\int_{\Omega}{ #1} \de \f x}
\newcommand{\ov}[1]{\overline{#1}}
\begin{document}
\title{\Large\bfseries Existence of weak solutions to the Ericksen--Leslie model
\newline
for a general class of free energies\footnote{This work was funded by CRC 901 {\em Control of self-organizing nonlinear systems: Theoretical methods and concepts of application}\/.
 (Project A8)}}
\author{Etienne Emmrich\thanks{%
        Technische Universit\"{a}t Berlin,
        Institut f\"{u}r Mathematik,
        Stra{\ss}e des 17.~Juni 136,
        10623 Berlin, Germany
        \newline{\tt emmrich@math.tu-berlin.de}
        }
\and    Robert Lasarzik\thanks{%
        Technische Universit\"{a}t Berlin,
        Institut f\"{u}r Mathematik,
        Stra{\ss}e des 17.~Juni 136,
        10623 Berlin, Germany
        \newline{\tt lasarzik@math.tu-berlin.de}
        }
}%
\markboth{Weak solutions to the Ericksen--Leslie model}{E.~Emmrich and R.~Lasarzik}
\date{ }
\maketitle
\begin{abstract} A quasistatic model due to Ericksen and Leslie describing incompressible liquid crystals is studied for a general class of free energies. Global existence of weak solutions is proven via a Galerkin approximation with eigenfunctions of a strongly elliptic operator.
A novelty is that the principal part of the differential operator appearing in the director equation can be nonlinear.
\end{abstract}
\setcounter{tocdepth}{2}
\tableofcontents
\section{Introduction}\label{sec:int}
Liquid crystals are fluids that exhibit anisotropic (directional depending) properties.
After several reports on such materials in the second half of the 19th century
 (see~Heinz\cite{heinz}, Virchow~\cite{virchow}, and Reinitzer~\cite{reinitzer}), they were first named and characterised as a
  state of matter in between fluids and solids by Otto Lehmann (see~\cite{lehmann}). Only in the
   second half of the last century liquid crystals came into the focus for many
    applications~(see~Palffy-Muhoray~\cite{peter}), where the liquid crystal display
     (see~{Heilmeier} et\,al.~\cite{heilmeier}) is only the most prominent one.

This article is mainly concerned with nematic liquid crystals, which is a special mesophase of
 liquid crystals. In this phase, the rod-like molecules are randomly distributed in space as in
  isotropic liquids. These molecules tend to align in a common direction, even though they are not
   in a rigid lattice structure as in solids (see Figure~\ref{fig}).
\begin{figure}[ht]
\begin{center}
\begin{pspicture}(0,0.3)(10,3.5)
\rput(1.5,3.5){%
\begin{minipage}{3cm}%
\begin{center}%
\textbf{isotropic fluid}
\end{center}
\end{minipage}}
\pspolygon[linecolor=matheonblue](0,0)(0,3)(3,3)(3,0)
\rput{90}(.4,.2){\psellipse*[linecolor=matheonlightblue](0,0)(.1,.25)}
\rput{45}(1,.2){\psellipse*[linecolor=matheonlightblue](0,0)(.1,.25)}
\rput{-33}(.2,.6){\psellipse*[linecolor=matheonlightblue](0,0)(.1,.25)}
\rput{163}(.5,.7){\psellipse*[linecolor=matheonlightblue](0,0)(.1,.25)}
\rput{50}(1.4,0.4){\psellipse*[linecolor=matheonlightblue](0,0)(.1,.25)}
\rput{-60}(2.3,0.2){\psellipse*[linecolor=matheonlightblue](0,0)(.1,.25)}
\rput{-45}(1,0.7){\psellipse*[linecolor=matheonlightblue](0,0)(.1,.25)}
\rput{-90}(1.8,0.7){\psellipse*[linecolor=matheonlightblue](0,0)(.1,.25)}
\rput{35}(2.7,0.6){\psellipse*[linecolor=matheonlightblue](0,0)(.1,.25)}
\rput{-20}(0.3,1.2){\psellipse*[linecolor=matheonlightblue](0,0)(.1,.25)}
\rput{60}(0.8,1.3){\psellipse*[linecolor=matheonlightblue](0,0)(.1,.25)}
\rput{45}(.215,1.7){\psellipse*[linecolor=matheonlightblue](0,0)(.1,.25)}
\rput{-80}(0.7,1.8){\psellipse*[linecolor=matheonlightblue](0,0)(.1,.25)}
\rput{-35}(1.4,1.4){\psellipse*[linecolor=matheonlightblue](0,0)(.1,.25)}
\rput{-70}(1.6,1.1){\psellipse*[linecolor=matheonlightblue](0,0)(.1,.25)}
\rput{-45}(2.4,1){\psellipse*[linecolor=matheonlightblue](0,0)(.1,.25)}
\rput{0}(2.8,1.3){\psellipse*[linecolor=matheonlightblue](0,0)(.1,.25)}
\rput{90}(1.5,1.9){\psellipse*[linecolor=matheonlightblue](0,0)(.1,.25)}
\rput{40}(2.4,1.8){\psellipse*[linecolor=matheonlightblue](0,0)(.1,.25)}
\rput{30}(0.3,2.6){\psellipse*[linecolor=matheonlightblue](0,0)(.1,.25)}
\rput{60}(0.8,1.3){\psellipse*[linecolor=matheonlightblue](0,0)(.1,.25)}
\rput{-5}(0.8,2.5){\psellipse*[linecolor=matheonlightblue](0,0)(.1,.25)}
\rput{-40}(1.4,2.5){\psellipse*[linecolor=matheonlightblue](0,0)(.1,.25)}
\rput{0}(2.0,2.6){\psellipse*[linecolor=matheonlightblue](0,0)(.1,.25)}
\rput{-60}(2.6,2.7){\psellipse*[linecolor=matheonlightblue](0,0)(.1,.25)}
\rput{-85}(2.7,2.2){\psellipse*[linecolor=matheonlightblue](0,0)(.1,.25)}
\rput(5,3.5){\begin{minipage}{3cm}\begin{center}%
\textbf{{nematic liquid crystal}}
\end{center}
\end{minipage}}

\pspolygon[linecolor=matheonblue](3.5,0)(3.5,3.0)(6.5,3)(6.5,0)
\rput{10}(3.7,.3){\psellipse*[linecolor=matheonlightblue](0,0)(.1,.25)}
\rput{9}(4.0,.4){\psellipse*[linecolor=matheonlightblue](0,0)(.1,.25)}
\rput{-5}(4.4,.6){\psellipse*[linecolor=matheonlightblue](0,0)(.1,.25)}
\rput{0}(4.6,.3){\psellipse*[linecolor=matheonlightblue](0,0)(.1,.25)}
\rput{3}(5.0,.5){\psellipse*[linecolor=matheonlightblue](0,0)(.1,.25)}
\rput{-2}(5.4,.3){\psellipse*[linecolor=matheonlightblue](0,0)(.1,.25)}
\rput{7}(5.7,.7){\psellipse*[linecolor=matheonlightblue](0,0)(.1,.25)}
\rput{0}(6.2,.6){\psellipse*[linecolor=matheonlightblue](0,0)(.1,.25)}
\rput{-10}(3.6,.9){\psellipse*[linecolor=matheonlightblue](0,0)(.1,.25)}
\rput{6}(4.2,1.1){\psellipse*[linecolor=matheonlightblue](0,0)(.1,.25)}
\rput{3}(4.65,.8){\psellipse*[linecolor=matheonlightblue](0,0)(.1,.25)}
\rput{-4}(4.85,1.2){\psellipse*[linecolor=matheonlightblue](0,0)(.1,.25)}
\rput{9}(5.1,1.0){\psellipse*[linecolor=matheonlightblue](0,0)(.1,.25)}
\rput{0}(5.4,1.2){\psellipse*[linecolor=matheonlightblue](0,0)(.1,.25)}
\rput{2}(5.65,1.4){\psellipse*[linecolor=matheonlightblue](0,0)(.1,.25)}
\rput{-5}(5.9,1.0){\psellipse*[linecolor=matheonlightblue](0,0)(.1,.25)}
\rput{6}(6.2,1.3){\psellipse*[linecolor=matheonlightblue](0,0)(.1,.25)}
\rput{-8}(6.3,2.0){\psellipse*[linecolor=matheonlightblue](0,0)(.1,.25)}
\rput{0}(3.7,1.6){\psellipse*[linecolor=matheonlightblue](0,0)(.1,.25)}
\rput{-20}(4,1.4){\psellipse*[linecolor=matheonlightblue](0,0)(.1,.25)}
\rput{-10}(4.5,1.3){\psellipse*[linecolor=matheonlightblue](0,0)(.1,.25)}
\rput{20}(3.8,2.2){\psellipse*[linecolor=matheonlightblue](0,0)(.1,.25)}
\rput{10}(4.5,1.9){\psellipse*[linecolor=matheonlightblue](0,0)(.1,.25)}
\rput{-5}(4.9,1.8){\psellipse*[linecolor=matheonlightblue](0,0)(.1,.25)}
\rput{0}(5.4,1.9){\psellipse*[linecolor=matheonlightblue](0,0)(.1,.25)}
\rput{-10}(5.9,1.6){\psellipse*[linecolor=matheonlightblue](0,0)(.1,.25)}
\rput{-20}(3.8,2.7){\psellipse*[linecolor=matheonlightblue](0,0)(.1,.25)}
\rput{-10}(4.1,2.1){\psellipse*[linecolor=matheonlightblue](0,0)(.1,.25)}
\rput{10}(4.3,2.7){\psellipse*[linecolor=matheonlightblue](0,0)(.1,.25)}
\rput{-5}(4.7,2.3){\psellipse*[linecolor=matheonlightblue](0,0)(.1,.25)}
\rput{15}(5.0,2.6){\psellipse*[linecolor=matheonlightblue](0,0)(.1,.25)}
\rput{-10}(5.15,2.1){\psellipse*[linecolor=matheonlightblue](0,0)(.1,.25)}
\rput{5}(5.4,2.6){\psellipse*[linecolor=matheonlightblue](0,0)(.1,.25)}
\rput{-7}(5.75,2.75){\psellipse*[linecolor=matheonlightblue](0,0)(.1,.25)}
\rput{-15}(5.7,2.1){\psellipse*[linecolor=matheonlightblue](0,0)(.1,.25)}
\rput{15}(6.3,2.6){\psellipse*[linecolor=matheonlightblue](0,0)(.1,.25)}

\rput(8.5,3.5){%
\begin{minipage}{3cm}%
\begin{center}%
{\textbf{crystal}}
\end{center}
\end{minipage}
}
\pspolygon[linecolor=matheonblue](7,0)(7,3)(10,3)(10,0)
\multido{\rr=0.3+0.6}{5}{
\multido{\i=0+1,\r=7.15+0.3}{10}{%
\rput{0}(\r,\rr){\psellipse*[linecolor=matheonlightblue](0,0)(.1,.25)}%
}%
}
\end{pspicture}
\end{center}\caption{%
Ordering of the molecules in a nematic liquid crystal in comparison to isotropic liquids an crystals
\label{fig}}
\end{figure}
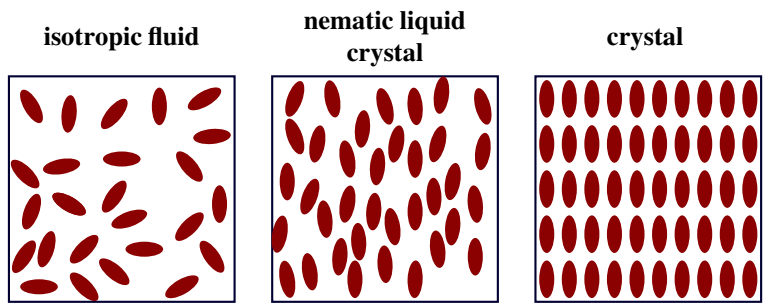
The most common model describes the anisotropic properties, i.\,e.,~the preferred direction of the
 alignment, by a unit vector field~$\f d$ and the fluid flow by a velocity field~$\f v$.
This model is due  to Oseen~\cite{oseen} and Frank~\cite{frank} in the stationary case and Ericksen~\cite{Erick1,Erick2} and Leslie~\cite{leslie} in the non-stationary case.
This model is indeed quite successful and has also been confirmed by experiments (see  Beris and Edwards~\cite[\S 11.1 page 463]{beris}).


In this article, we prove existence of weak solutions to the Ericksen--Leslie model under rather weak assumptions on the free energy.

\subsection{Review of known results}
 Ericksen~\cite{Erick1,Erick2} and Leslie~\cite{leslie} introduced the following system, which consists of an  equation for the
evolution of the anisotropic parameter $\f d$ coupled with an incompressible Navier--Stokes equation for the velocity  $\f v$ and the pressure $p$
with a certain additional
stress tensor:
\begin{subequations}
\label{ELeq}
\begin{align}
\rho  \br{\f v} + \nabla p + \di \left ( \nabla \f d ^T \pat{F}{\nabla \f d }  \right ) - \di \f T^L& =\rho \f g \, , \label{EL2}\\
\rho_1 \frac{\de^2 \f d}{\de t^2}  - \di \left ( \pat{F}{\nabla \f d} \right )+  \pat{F}{\f d}  - \lambda_1 \left ( \frac{\de \f d}{\de t}- \sk{v}\f d  \right ) - \lambda_2 \sy{v} \f d &= \rho_1\f f \, ,\label{EL1}\\
\di \f v &=0\, ,\label{EL3}\\
|\f d|^2 &=1\,.\label{EL4}
\end{align}%
\end{subequations}
Here $\br{}:= \partial_t+(\f v \cdot \nabla ) $ denotes the material derivative. The free energy density is denoted by $F=F(\f d, \nabla \f d)$.
Moreover, $\f f$ and $\f g$ represent external forces
acting on the director and on the velocity field, respectively. Finally, $\rho$ denotes the mass density of the fluid whereas $\rho_1 = \rho \bar{r}^2$ with $\bar{r}$ being of the size of the length of the molecules.

The dissipative part
of the stress tensor, also called Leslie stress, is given by
\begin{align}
\begin{split}
\f T^L:= &\mu_1 (\f d \cdot \sy v \f d )\f d\otimes \f  d + \mu_2 \f e \otimes \f d  + \mu_3  \f d \otimes \f e \\&+ \mu_4 \sy{v} + \mu_5
\sy v \f d \otimes\f d+ \mu_6 \f d \otimes \sy v  \f d .
\end{split}\label{stressE}
\end{align}
Here we use the abbreviation $\f e:=  \frac{\de \f d}{\de t}- \sk v \f d$. Note that $\f e$ is an invariant of the system with respect to translations and rotations (see Leslie~\cite{leslie}).
The constants appearing in  \eqref{ELeq} and  \eqref{stressE} are
related to each other in the following way (see Leslie~\cite{leslie}):
\begin{subequations}
\begin{align}\label{symcon}
\lambda_1 = \mu_2-\mu_3, \qquad \lambda_2 = \mu_5-\mu_6\,.
\end{align}

Additionally, Parodi's relation
\begin{align}
 \lambda_2 + \mu_2 + \mu _3 = 0
\label{parodi}
\end{align}%
\end{subequations}
is often assumed to hold~(see Lin and Liu~\cite{linliu3} or~Parodi~\cite{parodi}), but will not be assumed to hold for the proof of our existence result.  Parodi's relation follows from Onsager's reciprocal relation and
can be employed in order to obtain the stresses as derivatives of a dissipation distance (see~ Wu, Xu and Liu~\cite{hiu}).


The first mathematical analysis of the Ericksen--Leslie model is due to
Lin and Liu~\cite{linliu1} for the simplified model
\begin{align}
\begin{split}
\t \f v+ ( \f v \cdot \nabla )\f v + \nabla p - \mu_4 \Delta \f v
= & -  \di ( \nabla \f d ^T \nabla \f d) \,,\\
\t \f d + ( \f v \cdot \nabla )\f d  -  \Delta \f d + \frac{1}{\varepsilon}(| \f d |^2-1)\f d = &  0\,,\\  \di \f v =&0\,.
\end{split}\label{simplified}
\end{align}
The norm restriction~\eqref{EL4} is incorporated by a standard relaxation technique adding a double-well potential to the free energy.
The free energy potential for this model is given by
\begin{align}
F_\varepsilon ( \f d , \nabla \f d) = \frac{1}{2} |\nabla \f d |^2 + \frac{1}{4\varepsilon^2} ( | \f d |^2 -1) ^2 , \quad \varepsilon >0 \, .\label{penal}
\end{align}

Indeed, Lin and Liu~\cite{linliu1} prove global existence of weak solutions as well as local existence of strong solutions to~\eqref{simplified}.
In~\cite{linliu3}, the authors  generalise these results to the system~\eqref{EL1}-\eqref{EL3} equipped with the energy~\eqref{penal} and under the additional assumption $\rho_1=\lambda_2=0$.
With this last simplification ($\lambda_2=0$), translational forces of the fluid onto the director are neglected. This enables the authors to prove a weak maximum principle
which is  essential for the analysis in \cite{linliu1} and~\cite{linliu3}.

The full Ericksen--Leslie model \eqref{EL2}-\eqref{EL3} (with $\rho_1=0$) equipped with the Dirichlet energy and double-well potential~\eqref{penal} was considered by Wu, Xu and Liu~\cite{hiu}. They show existence of weak solutions under the condition that $\mu_4$ is large enough.
Cavaterra, Rocca and Wu~\cite{allgemein} prove the existence of weak solutions for the same system when $\mu_4$ is only assumed to be positive. They add a regularising $p$-Laplacian to the velocity equation.
Feireisl et\,al.~\cite{nonisothermal} generalised the Ericksen--Leslie model to account for nonisothermal effects by considering additionally to system~\eqref{ELeq} an energy balance and an entropy inequality. They show global existence of weak solutions.

There are also several articles studying the local well-posedness of the Ericksen--Leslie model.
Wang, Zhang and Zhang~\cite{recent} show local existence of strong solutions to system~\eqref{ELeq} equipped with the Dirichlet energy $F(\f d , \nabla \f d ) = | \nabla \f d |^2$, where equation~\eqref{EL1} is replaced by
\begin{align}
 \f d \times \left (- \di \left ( \pat{F}{\nabla \f d} \right )+  \pat{F}{\f d}  - \lambda_1 \left ( \frac{\de \f d}{\de t}- \sk{v}\f d  \right ) - \lambda_2 \sy{v} \f d \right )= 0 \label{EL5}\,.
\end{align}
Taking the equation for the director in the cross product with the director itself, assures that the norm restriction~\eqref{EL4} is satisfied for the whole evolution. This does not need to be the case for the general Ericksen--Leslie model~\eqref{EL2}-\eqref{EL3} with energy~\eqref{penal} and $\rho_1=0$.
Another approach is due to Pr\"u\ss{} and co-authors introducing a
thermodynamically consistent system~\cite{Pruessnematic} and proving
local existence and stability results via a semigroup approach for quasilinear equations (see Hieber and Pr\"u\ss{}~\cite{Pruessnematic}  and Hieber et\,al.~\cite{Pruess2}).
The simplified model~\eqref{simplified} with the director equation taken in the cross-product with $\f d$ and equipped with the so-called Oseen--Frank energy~\eqref{oseen} below was considered by Hong, Li, and Xin~\cite{localin3d}, they managed to prove the local existence of strong solutions.

For a broader overview of results concerning the analysis of liquid crystal models, we refer to Lin and Liu~\cite{linliu4} and Lin and Wang~\cite{recentDevelopments}.

\subsection{Free energy potential}

The free energy potential $F$ models
the inner forces and thus the influence of the molecules onto each other as well as on the velocity field.
The focus of the present work is to generalise the global existence theory available for the Ericksen--Leslie model to a larger class of free energies, including also potentials associated to nonlinear principal parts in the director equation.

To model distortions in the material, already Leslie (see~\cite{leslie}) suggested
to consider  the free energy potential due to Oseen~\cite{oseen} and Frank~\cite{frank},  called Oseen--Frank
energy,
\begin{align}\label{oseen}
F:= k_1 (\di \f d )^2 + k_2 ( \f d \cdot (\curl \f d))^2 + k_3 | \f d
\times (\curl \f d) |^2 + \alpha \left ( \tr (\nabla \f  d^2 ) - (\di \f d )^2\right ).
\end{align}
Note that $\f d\times (\curl \f  d ) = - 2\sk  d \f  d  $. The last term of~\eqref{oseen} can be expressed as the divergence of a vector field, $\di(\nabla \f d \f d - ( \di \f d ) \f d) =\tr (\nabla \f  d ^2) - (\di \f d )^2$, and with Gau\ss{}' formula it is already prescribed by the boundary data.
With the one-constant approximation $ k_1=k_2=k_3=\alpha $ and employing $|\f d| = 1$, one obtains the
Dirichlet energy $F(\f d , \nabla \f d)= k_1 | \nabla \f d |^2 $. This gives rise to study the energy potential~\eqref{penal}.

In the physics literature, there are several choices of free energy potentials, which are not covered by the available mathematical existence theory of generalised solutions yet.
Possible electromagnetic field effects could be taken into account by considering (see de Gennes~\cite{gennes})
\begin{equation}
F_H (\f d , \nabla \f d) := F( \f d , \nabla \f d) - \chi_\bot |\f H|^2  - (\chi_\|-\chi_\bot) ( \f d \cdot \f H )^2\,.\label{FH}
\end{equation}
Here $\f H$ denotes the electromagnetic field and $\chi_\bot$ and $\chi_\|$ are the magnetic susceptibility constants for a magnetic field parallel and perpendicular to the director, respectively.
Already Leslie suggests to incorporate two additional degrees of freedom into the system, which can be achieved by considering a free energy potential of the form
\begin{align}
F_A( \f d , \nabla \f d) : =  F( \f d, \nabla \f d) - \f d\cdot \nabla\f d \f b + \frac{\bar{b}}{2}|\f d|^2 \, \label{FA}
\end{align}
with $ \f b \in \R^3 $ and $\overline{b}\in\R$.
Furthermore, the case of the following simplified Oseen--Frank energy is not fully treated in the literature yet.
For $k_2=k_3$ and under the assumption $| \f d |=1$, the Oseen--Frank energy can be transformed to (see Section~\ref{sec:ex})
\begin{align*}
F(\f d ,\nabla \f d ) = k_1 ( \di \f d )^2 +k_2 | \curl \f d|^2 + \frac{1}{\varepsilon}(| \f d|^2-1)^2\,
\end{align*}
with $k_1,k_2>0$.

It is also possible to prove the existence of weak solutions to  the Ericksen--Leslie system equipped with a scaled version of the Oseen--Frank energy.
This energy is given by
\begin{align}\label{scaledO}
\begin{split}
F(\f d , \nabla \f d ) :  = {}&\frac{k_1}{2} (\di \f d )^2 +\frac{k_2}{2}| \curl \f d |^2 \\ &+  ( 1+ | \nabla \f d|^2)^{-s}( 1+ | \f d|^2)^{-1} \left ( \frac{k_3}{2} ( \f d \cdot \curl \f d) ^2 + \frac{ k_4}{2} | \f d \times \curl \f d |^2 \right ) \,
\end{split}
\end{align}
with $k_1,k_2>0$ and $k_3,k_4$ sufficiently small as well as $s>1 /6 $.
This is a modification of the Oseen--Frank energy taking into account the non-quadratic terms and thus anisotropic, director-depending properties of the material.
The non-quadratic parts of the free energy are scaled appropriately and the energy has an anisotropic character comparable to the Oseen--Frank energy.

We provide the proof of existence of weak solutions to the Ericksen--Leslie equation~\eqref{eq:strong} equipped with each of the above physical relevant energies, except for the general Oseen--Frank energy~\eqref{oseen}.
For the existence of measure-valued solutions to the problem with general Oseen--Frank energy, we refer to~\cite{unseremasse}.

\subsection{Outline of the paper}
In the present paper, we study the original Ericksen--Leslie system~\eqref{ELeq} in three dimensions with $\rho_1 = 0$ (macroscopic theory)
together with a relaxation by a  double-well potential.
We focus in particular on the free energy and introduce a class of free energy functions that allow us to show global existence of weak solutions.
The class of free energies we consider is of the type
\begin{equation}
F(\f d, \nabla \f d) = \frac{1}{2} \, \nabla \f d : \f \Lambda : \nabla \f d + \tilde{F}(\f d, \nabla \f d) \, ,\label{tildeF}
\end{equation}
where $\f \Lambda$ denotes a tensor of fourth order and $\tilde{F}$ collects terms that are of lower order with respect to $\nabla \f d$.
This class of free energies includes, for instance, all free energy potentials mentioned above except  the general Oseen--Frank energy (see Section~\ref{sec:ex}).

In order to ensure dissipativity of the system, we require that (see also the equivalent formulation (\ref{con}) below)

\begin{gather*}
\mu_1 > 0 \, , \; \mu_3 > \mu_2 \, , \; \mu_4 > 0 \, , \;
(\mu_3 - \mu_2)(\mu_6 + \mu_5) >(\mu_3 + \mu_2)(\mu_6 - \mu_5) \, , \;
\\
4 (\mu_3 - \mu_2)(\mu_6 + \mu_5) > \left( (\mu_3 + \mu_2) + (\mu_6 - \mu_5) \right)^2 .
\end{gather*}

For proving existence of a solution,
we employ a Galerkin method to approximate both equations (\ref{eq:velo}), (\ref{eq:dir}) simultaneously. This is in contrast to previous work such as Wu, Xu and Liu~\cite{hiu} or
Cavaterra, Rocca and Wu~\cite{allgemein} where the authors combine a Schauder fixed point argument with a Galerkin approximation of only the Navier--Stokes-like equation and solving the director equation in each step exactly.
This method relies on existence and continuity of the  solution operator to equation~\eqref{EL1}. To be able to use such a property previous work had to invoke additional regularity in the approximation of the velocity field either by assuming $\mu_4$ to be sufficiently large~\cite{hiu} or by introducing an additional regularisation~\cite{allgemein}.
Due to the generalisation with respect to the free energy considered in the present paper, the continuity of the solution operator to equation~\eqref{EL1} is no longer at hand.
Additionally, a simultaneous discretisation is more suitable for a numerical approximation.

The paper is organised as follows: In Section~\ref{pre}, we collect some notation and important inequalities. Section~\ref{sec:model} then contains the main result together with a detailed description of the class of free energies. The proof is then carried out in Section~\ref{Sec:Proof}.
In Section~\ref{sec:more}, we generalise the result to possible nonlinear principal parts and comment on the adaptations needed in the proof.
 Finally, some examples are discussed in Section~\ref{sec:ex}.
\section{Preliminaries}\label{pre}
\subsection{Notation}\label{sec:not}
Vectors of $\R^3$ are denoted by bold small Latin letters. Matrices of $\R^{3\times 3}$ are denoted by bold capital Latin letters. We also need tensors of third and fourth order, which are denoted by capital and bold capital Greek letters, respectively.
Moreover, numbers are denoted be small Latin or Greek letters, and capital Latin letters are reserved for potentials.

The Euclidean inner product in $\R^3$ is denoted by a dot,
$ \f a \cdot \f b : = \f a ^T \f b = \sum_{i=1}^3 \f a_i \f b_i$  for $ \f a, \f b \in \R^3$.
The Frobenius inner product in the space $\R^{3\times 3}$ of matrices is denoted by a double dot, $ \f A: \f B:= \tr ( \f A^T \f B)= \sum_{i,j=1}^3 \f A_{ij} \f B_{ij}$ for $\f A , \f B \in \R^{3\times 3}$.
We also employ the corresponding Euclidean norm with $| \f a|^2 = \f a \cdot \f a$ for $ \f a \in \R^3$ and Frobenius norm with $ |\f A|^2=\f A:\f A$ for $\f A \in \R^{3\times 3}$.
In addition, we define products of tensors of different order in a similar fashion: The product of a third with a second order tensor is defined by
\begin{align*}
\Gamma : \f A := \left [ \sum_{j,k=1}^3 \Gamma_{ijk}\f A_{jk}\right ]_{i=1}^3\, ,  \quad  \Gamma\in \R^{3\times 3\times 3} , \, \f A \in \R^{3\times 3} .
\end{align*}
The product of a fourth order with a second order tensor is defined by
\begin{align*}
\f \Gamma : \f A : =\left [ \sum_{k,l=1} ^3 \f \Gamma_{ijkl} \f A_{kl}\right ]_{i,j=1}^3, \quad    \f \Gamma \in \R^{3\times 3 \times 3 \times 3},  \,  \f A \in \R^{3\times 3 }  .
\end{align*}
The product of a fourth order with a third order tensor is defined by
\begin{align*}
\f \Gamma \drei \Gamma : =\left [ \sum_{j,k,l=1} ^3 \f \Gamma_{ijkl} \Gamma_{jkl}\right ]_{i=1}^3, \quad    \f \Gamma \in \R^{3\times 3 \times 3 \times 3},  \,  \Gamma \in \R^{3\times 3 \times 3}  .
\end{align*}
The standard matrix and matrix-vector multiplication, however, is written without an extra sign for bre\-vi\-ty,
$$\f A \f B =\left [ \sum _{j=1}^3 \f A_{ij}\f B_{jk} \right ]_{i,k=1}^3 \,, \quad  \f A \f a = \left [ \sum _{j=1}^3 \f A_{ij}\f a_j \right ]_{i=1}^3\, , \quad  \f A \in \R^{3\times 3},\,\f B \in \R^{3\times3} ,\, \f a \in \R^3 .$$
The outer product is denoted by
$\f a \otimes \f b = \f a \f b^T = \left [ \f a_i  \f b_j\right ]_{i,j=1}^3$ for $\f a , \f b \in \R^3$. Note that
$\tr (\f a \otimes \f b  ) = \f a\cdot \f b$.
The symmetric and skew-symmetric part of a matrix are denoted by $\f A_{\sym}: = \frac{1}{2} (\f A + \f A^T)$ and
$\f A _{\skw} : = \frac{1}{2}( \f A - \f A^T)$ for $\f A \in \R^{3\times  3}$, respectively. For the Frobenius product of two matrices $\f A, \f B \in \R^{3\times 3 }$, we find that
 \begin{align*}
 \f A: \f B = \f A : \f B_{\sym}  \text{ if } \f A^T= \f A\, ,\quad
  \f A: \f B = \f A : \f B_{\skw} \text{ if } \f A^T= -\f A\, .
 \end{align*}
Moreover, there holds $\f A^T\f B : \f C = \f B : \f A \f C$ for
$\f A, \f B, \f C \in \R^{3\times 3}$ as well as
$ \f a\otimes \f b : \f A = \f a \cdot \f A \f b$ for
$\f a, \f b \in \R^3$, $\f A \in \R^{3\times 3 }$. This implies
$ \f a \otimes \f a : \f A = \f a \cdot \f A \f a =  \f a \cdot \f A_{\sym} \f a$.

We use  the Nabla symbol $\nabla $  for real-valued functions $f : \R^3 \to \R$, vector-valued functions $ \f f : \R^3 \to \R^3$ as well as matrix-valued functions $\f A : \R^3 \to \R^{3\times 3}$ denoting
\begin{align*}
\nabla f := \left [ \pat{f}{\f x_i} \right ] _{i=1}^3\, ,\quad
\nabla \f f  := \left [ \pat{\f f _i}{ \f x_j} \right ] _{i,j=1}^3 \, ,\quad
\nabla \f A  := \left [ \pat{\f A _{ij}}{ \f x_k} \right ] _{i,j,k=1}^3\, .
\end{align*}
For brevity, we write $ \nabla \f f^T $ instead of $ ( \nabla \f f)^T$. The divergence of a vector-valued and a matrix-valued function is defined by
\begin{align*}
\di \f f := \sum_{i=1}^3 \pat{\f f _i}{\f x_i} = \tr ( \nabla \f f)\, , \quad  \di \f A := \left [\sum_{j=1}^3 \pat{\f A_{ij}}{\f x_j}\right] _{i=1}^3\, .
\end{align*}
The symmetric and skew-symmetric part of the gradient of a vector-valued function is denoted by $ \sy f$ and $ \sk f$, respectively.
Note that $( \f v\cdot \nabla ) \f f = ( \nabla \f f) \f v = \nabla \f f\, \f v $.

Throughout this paper, let $\Omega \subset \R^3$ be a bounded domain of class $\C^{2}$.
We rely on the usual notation for spaces of continuous functions, Lebesgue and Sobolev spaces. Spaces of vector-valued functions are  emphasised by bold letters, for example
$
\f L^p(\Omega) := L^p(\Omega; \R^3)$,
$\f W^{k,p}(\Omega) := W^{k,p}(\Omega; \R^3)$.
If it is clear from the context, we also use this bold notation for spaces of matrix-valued functions.
For brevity, we often omit calling the domain $\Omega$.
The standard inner product in $L^2 ( \Omega; \R^3)$ is just denoted by
$ (\cdot \, , \cdot )$ and in $L^2 ( \Omega ; \R^{3\times 3 })$
by $(\cdot : \cdot )$.
In view of the Poincar\'e{}--Friedrichs inequality, we equip $\Hb$ with the norm $\| \cdot \|_{\Hb}:= \| \nabla \cdot \|_{\Le} $.
We often need to work with the space $\Hc$ that we equip with the norm $\|\cdot\|_{\Hc} := \| \Delta \cdot\|_{\Le}$, which is equivalent to the standard $\f H^2$-norm (see Gilbarg and Trudinger~\cite[Thm.~9.15, Lemma~9.17]{gil} or Ladyzhenskaya~\cite[p.~18f.]{LadyFluid}).

The space of smooth solenoidal functions with compact support is denoted by $\mathcal{C}_{c,\sigma}^\infty(\Omega;\R^3)$. By $\f L^p_{\sigma}( \Omega) $, $\V(\Omega)$,  and $ \f W^{1,p}_{0,\sigma}( \Omega)$, we denote the closure of $\mathcal{C}_{c,\sigma}^\infty(\Omega;\R^3)$ with respect to the norm of $\f L^p(\Omega) $, $ \f H^1( \Omega) $, and $ \f W^{1,p}(\Omega)$, respectively.

The dual space of a Banach space $V$ is always denoted by $ V^*$ and equipped with the standard norm; the duality pairing is denoted by $\langle\cdot, \cdot \rangle$. The duality pairing between $\f L^p(\Omega)$ and $\f L^q(\Omega)$ (with $1/p+1/q=1$), however, is denoted by $(\cdot , \cdot )$ or $( \cdot : \cdot )$. The dual of $\f H^1_0$ is denoted by $\f H^{-1}$.

The Banach space of linear bounded operators mapping a Banach space $V$ into itself is denoted by $\mathcal{L}(V)$ and equipped with the usual norm.
For a given Banach space $ V$, Bochner--Lebesgue spaces are denoted, as usual,  by $ L^p(0,T; V)$. Moreover,  $W^{1,p}(0,T; V)$ denotes the Banach space of abstract functions in $ L^p(0,T; V)$ whose weak time derivative exists and is again in $ L^p(0,T; V)$ (see also
Diestel and Uhl~\cite[Section~II.2]{diestel} or
Roub\'i\v{c}ek~\cite[Section~1.5]{roubicek} for more details).
We often omit the time interval $(0,T)$ and the domain $\Omega$ and just write, e.g., $L^p(\f W^{k,p})$ for brevity.
By $ \AC ( [0,T]; V)$, $\C([0,T]; V) $, and $ \C_w([0,T]; V)$, we denote the spaces of abstract functions mapping $[0,T]$ into $V$ that are absolutely continuous, continuous, and continuous with respect to the weak topology in $V$, respectively.

Finally, by $c>0$, we denote a generic positive constant.

\subsection{A few interpolation inequalities}
We commence with a Gagliardo--Nirenberg-type result for time-dependent functions.
\begin{lemma}
\label{lem:nir}
Let $p\in [2,6]$, $q\in [6,\infty]$ and let $r,s \in [1 , \infty]$ and
$\theta_1 , \theta_2 \in [0,2]$ such that
\begin{align}
  \frac 1 p=\frac 1 2-\frac{\theta_1}{3r} \, , \quad
 \frac 1q = \frac 16 -\frac{\theta_2}{3s} \, .
\label{bedp}
\end{align}
Then there exists a constant $c>0$ such that the inequalities
\begin{align*}
\|\nabla  \f d\|_{L^r(\f L^{p})}^r  \leq c \|\f d\|_{L^2(\Hc)}^{\theta_1} \|\f d\|_{L^\infty(\Hb)}^{r-\theta_1} \, ,
\quad
\| \f d\|_{L^s(\f L^q)}^s \leq c \|\f d\|_{L^2(\Hc)}^{\theta_2} \|\f d\|_{L^\infty(\Hb)}^{s-\theta_2}
\end{align*}
hold for any $\f d \in L^\infty(0,T;\f H^1_0) \cap L^2(0,T;\Hc)$.
\end{lemma}

\begin{proof}
Since $\Hc$ is continuously embedded in $\f W^{1,6}$, we immediately find
with H\"{o}lder's inequality that for $p\in [2,6]$ and any $\f d \in \Hc$
\begin{equation*}
\|\nabla \f d \|_{\f L^{p}} \leq \|\nabla \f d\|_{\f L^6}^{3/2-3/p}\|\nabla \f d\|_{\f L^2}^{3/p-1/2}
\le c \| \f d\|_{\Hc}^{3/2-3/p}\|\f d\|_{\Hb}^{3/p-1/2} \, .
\end{equation*}
From the classical Gagliardo--Nirenberg inequality
(see, e.g., Nirenberg~\cite{nirenex}, Friedman~\cite{friedman}, or Zeidler~\cite[Section 21.19]{zeidler2a}), we infer that there exists $c>0$ such that for $q\in [6,\infty]$ and
any $\f d \in \Hc$
\begin{align*}
\|\f d \|_{\f L^{q}} \leq c\| \f d\|_{\Hc}^{1/2 - 3/q}\|\f d\|_{\f L^6}^{1/2 +3/q}
\le c \| \f d\|_{\Hc}^{1/2 - 3/q}\|\f d\|_{\Hb}^{1/2 + 3/q}
\, ;
\end{align*}
the case $q=\infty$ remains true with both exponents being $1/2$.

Let us now consider $\f d \in L^\infty(0,T;\f H^1_0) \cap L^2(0,T;\Hc)$ and integrate the foregoing estimates in time. We then find
\begin{equation*}
\|\nabla \f d\|_{L^r(\f L^{p})} \le c \|\f d\|^{3/2-3/p} _{L^{3r(1/2-1/p)}(\Hc)}      \| \f d\|_{L^{\infty}(\Hb)}^{3/p-1/2}
\end{equation*}
as long as $ \theta_1 = 3r(1/2-1/p) \le 2$ such that $L^2(0,T;\Hc)$ is continuously embedded in $L^{\theta_1}(0,T;\Hc)$. This proves the first inequality.

In the same fashion, one proves the second estimate.
\end{proof}

\begin{lemma}\label{cor:vreg}
There exists a constant $c>0$ such that for $p \in [2,6]$, $r\in [2,\infty]$ with
\begin{equation*}
\frac{1}{p} = \frac{1}{2} - \frac{2}{3r}
\end{equation*}
and any $\f v \in L^\infty(0,T;\f L^2)\cap L^2(0,T;\f H^1)$
\begin{align*}
 \|\f v\|_{L^r(\f L^{p})}^r \leq c \| \f v\|_{L^2(\He)}^2\|\f v\|_{L^\infty(\Le)}^{r-2}
 \, .
\end{align*}
\end{lemma}

\begin{proof}
The proof is analogous to that of Lemma~\ref{lem:nir}.
\end{proof}

\section{Ericksen--Leslie model and main result}\label{sec:model}
\subsection{Governing equations}
We consider the Ericksen--Leslie model~\eqref{ELeq} for dimensionless quantities with $\rho_1$ set to zero. We focus on a rather general class of free energy functions and incorporate the restriction of the director $\f d$ onto unit vectors into the free energy via a classical relaxation technique, see also~\eqref{penal}.
Furthermore, we restrict our considerations to the incompressible case with $\rho \equiv 1$. The governing equations then read as
\begin{subequations}\label{eq:strong}
\begin{align}
\t {\f v}  + ( \f v \cdot \nabla ) \f v + \nabla p + \di \f T^E- \di  \f T^L&= \f g, \label{nav}\\
\t {\f d }+ ( \f v \cdot \nabla ) \f d -\sk{v}\f d + \lambda \sy{v} \f d +\gamma \f q & =0,\label{dir}\\
\di \f v & = 0\, .
\end{align}%
\end{subequations}%

We recall that $\f v : \ov{\Omega}\times [0,T] \ra \R^3$ denotes the velocity  of the fluid, $\f d:\ov{\Omega}\times[0,T]\ra \R^3$ represents the orientation of the rod-like molecules, and $p:\ov{\Omega}\times [0,T] \ra\R$ denotes the pressure.
The Helmholtz free energy potential~$F$, which is described rigorously in the next section, is assumed to depend only on the director and its gradient, $F= F( \f d, \nabla \f d)$.
The free energy functional~$\mathcal{F}$  is defined by
\begin{align*}
\mathcal{F}: \He \ra \R , \quad \mathcal{F}(\f d):= \int_{\Omega} F( \f d, \nabla \f d) \de \f x \,,
\end{align*}
and $\f q$ is its variational derivative (see Furihata and Matsuo~\cite[Section 2.1]{furihata}),
\begin{subequations}\label{abkuerzungen}
\begin{align}\label{qdefq}
\f q :=\frac{\delta \mathcal{F}}{\delta \f d}(\f d) =  \pat{F}{\f d}(\f d , \nabla\f d)-\di \pat{F}{\nabla \f d}(\f d, \nabla \f d)\, .
\end{align}
The Ericksen stress tensor $\f T^E$ is given by
\begin{equation}
\f T^E = \nabla \f d^T \pat{F}{\nabla \f d}( \f d , \nabla\f d ) \, .\label{Erik}
\end{equation}
The Leslie tensor is given by
\begin{align}
\begin{split}
\f T^L ={}&  \mu_1 (\f d \cdot \sy{v}\f d )\f d \otimes \f d +\mu_4 \sy{v}
 + {(\mu_5+\mu_6)} \left (  \f d \otimes\sy{v}\f d \right )_{\sym}
\\
& +{(\mu_2+\mu_3)} \left (\f d \otimes \f e  \right )_{\sym}
 +\frac{\lambda}{\gamma}\left ( \f d \otimes \sy{v}\f d  \right )_{\skw} + \frac{1}{\gamma}\left (\f d \otimes \f e  \right )_{\skw}\, ,
\end{split}\label{Leslie}
\end{align}
where
\begin{align}
\f e : = \t {\f d} + ( \f v \cdot \nabla ) \f d - \sk v\f d\, .\label{e2}
\end{align}
This follows immediately from~\eqref{stressE}.
Following Walkington~\cite{wal1}, we have sorted the Leslie tensor~\eqref{stressE} into symmetric and skew symmetric parts.
We explicitly inserted~\eqref{symcon} established in Leslie~\cite{leslie}, and we set
\begin{align}
\gamma := - \frac{1}{ \lambda_1}= \frac{1}{\mu_3-\mu_2}, \qquad
\lambda :=  \frac{\lambda_2}{\lambda_1} = \gamma(\mu_6-\mu_5) . \label{constanten}
\end{align}
\end{subequations}
We emphasise that Parodi's law~\eqref{parodi} is neither essential for the reformulation nor the existence of weak solutions (see also
Wu et~al.~\cite{hiu}).

To ensure the dissipative character of the system, we assume that
\begin{align}
\begin{gathered}
\mu_1  > 0, \quad \mu_4 > 0,  \quad  \gamma > 0 ,\quad
\mu_5+\mu_6 - \lambda (\mu_2+\mu_3)>0    \, ,
\\
4 \gamma \big( \mu_5+\mu_6 - \lambda (\mu_2+\mu_3)\big)>
\big(\gamma(\mu_2+\mu_3) -\lambda\big)^2\,.
\end{gathered}\label{con}
\end{align}
The case $\mu_1=0$ can be dealt with similarly but somewhat simpler.

Finally, we impose boundary and initial conditions as follows:
\begin{subequations}\label{anfang}
\begin{align}
\f v(\f x, 0) &= \f v_0 (\f x) \quad\text{for } \f x \in \Omega ,
&\f v (  \f x, t ) &= \f 0  &\text{for }( \f x , t ) \in \partial \Omega
\times [0,T]
, \\
\f d (  \f x, 0 ) & = \f d_0 ( \f x) \quad\text{for } \f x \in \Omega ,
&\f d (  \f x ,t ) & = \f d_1 ( \f x )  &\text{for }(  \f x , t) \in  \partial \Omega \times [0,T]  .
\end{align}
\end{subequations}
We shall assume that $\f d_1= \f d_0$ on $\partial \Omega$, which is a compatibility condition providing regularity.

\subsection{A class of free energy potentials}\label{free}
This section is devoted to the free energy potential that describes the inner forces between the molecules.
We commence with a class of free energies that leads to a linear principal part. The more delicate case with a nonlinear principal part is dealt with in Section\ref{sec:more}.
Let us consider
\begin{equation}
F= F( \f h, \f S) \, \in \C^2 ( \R^3 \times \R^{3\times 3} ; \R) \,
\label{regF}%
\end{equation}
and let us assume that
\begin{subequations}\label{Lambda1}
\begin{equation}\label{LambdaS}
\frac{\partial^2 F}{\partial \f S^2} \equiv \f \Lambda \in \R^{3\times 3\times 3\times 3} \, ,
\end{equation}
where $\f \Lambda$ satisfies the symmetry condition
\begin{equation}\label{Lambdasym}
\f \Lambda_{ijkl} = \f \Lambda_{klij} \quad\text{for } i,j,k,l = 1,2,3 \,
\end{equation}
and the following strong Legendre--Hadamard (strong ellipticity) condition:
there exists $\eta > 0$ such that for all $\f a , \, \f b \in \R^3$
\begin{equation}\label{Lambdaellip}
\f a \otimes \f b : \f \Lambda : \f a \otimes \f b \ge \eta \,
|\f a |^2 |\f b|^2 \, .
\end{equation}
\end{subequations}
\begin{remark}
It is possible to generalise the assumptions on the second derivative of $F$ with respect to $\f S$. The tensor $\f \Lambda $ can continuously depend on the spatial variable $\f x$ (see Remark~\ref{rem:Lambda1}).
Additionally, a nonlinear term, which is sufficiently small, can be handled as a part of the second derivative of $F$ with respect to $\f S$ (see Section~\ref{sec:more}).
\end{remark}

In the course of the proof of our main result, we shall need further coercivity-type  assumptions on the free energy $F$ and its derivatives.
Let us assume that
there exist  $\eta_1>0$ and $\eta_2,\eta_3 \ge 0$ such that for all $\f h \in \R^3$ and $\f S \in \R^{3\times 3}$
\begin{align}
F(\f h ,\f S) \geq \eta_1|\f S|^2 -\eta_2 | \f h |^2-\eta_3 \, .\label{coerc1}
\end{align}
For a particular free energy, such a condition may follow from \eqref{Lambda1} together with suitable growth or nonnegativity assumptions on the lower order terms.

Later we will have that $F= F( \f d, \nabla \f d)$ with $\f d = \f d(\f x,t)$. Under the regularity assumption~\eqref{regF}, we may now consider (see also (\ref{qdefq}))
\begin{equation}\label{Edef}
\f q=   \pat{F}{\f h}(\f d, \nabla \f d)-\di \pat{F}{\f S}(\f d, \nabla \f d)
\, .
\end{equation}
With respect to $\f q$, we first observe that formally
\begin{align}
\f q  &= \pat{F}{\f h} (\f d, \nabla \f d) -\pat{^2 F}{\f{S}^2 } (\f d, \nabla \f d)\drei
\nabla (\nabla {\f d})^T - \pat{^2 F}{\f S\partial \f h}(\f d, \nabla \f d) : \nabla \f d^T
\nonumber
\\
&= \pat{F}{\f h} (\f d, \nabla \f d) - \di \left(  \f \Lambda :  \nabla \f d\right) - \pat{^2 F}{\f S\partial \f h}(\f d, \nabla \f d) : \nabla \f d^T
\label{QQQQ}
\, .
\end{align}
For arbitrary  $\f a,\f b,\f c \in \R^3$, one finds
\begin{align}\label{abcabc}
|\f a - \f b - \f c|^2 \geq \frac{1}{2}\, |\f b|^2 - 4|\f c|^2 - 4|\f a|^2\, ,
\end{align}
and thus
\begin{align*}
|\f q|^2\geq \frac{1}{2} \left | \di \left(  \f \Lambda :  \nabla \f d\right) \right  |^2 - 4 \left | \pat{^2 F}{\f S\partial \f h}(\f d, \nabla \f d) : \nabla \f d^T \right |^2 - 4 \left |\pat{F}{\f h} (\f d, \nabla \f d)\right |^2 \, .
\end{align*}

This calculation motivates the following growth conditions:
There exist $C_{\f S\f h}>0, C_{\f h} >0 $ and $\gamma_1 \in [ 2, 10/3)$, $\gamma_2 \in [ 6, 10)$ such that for all $\f h \in \R^3$ and $ \f S \in \R^{ 3\times 3} $
\begin{subequations}
\begin{align}
\left |\pat{^2F}{\f S\partial \f h } (\f h ,\f S)\right | & \leq C_{\f S\f h} \left (| \f S|^{\gamma_1/2-1}+ | \f h |^ { \gamma_3}+1\right )\label{FSh}\\
\left |\pat{F}{\f h  }(\f h ,\f S) \right |& \leq C_{ \f h}\left (|\f S|^{{\gamma_1}/{2}} + |\f h |^ {\gamma_2 /2} +1 \right ),
\label{FcoercE}
\end{align}
where
\begin{equation}\label{beta}
\gamma_3:=\frac{(\gamma_1-2)\gamma_2}{2\gamma_1} \, .
\end{equation}
This choice of exponents will allow us to derive appropriate a priori estimates.
Of course, the term with $|\f S|^{\gamma_1/2-1}$ in~\eqref{FSh} is superfluous for a potential fulfilling~\eqref{LambdaS} but will be essential for the analysis in Section~\ref{sec:more}.
\label{freeenergy}%
\end{subequations}

\subsection{Existence of weak solutions}
In this section, we state our main result on the existence of weak solutions. We  first give a precise definition of what we mean by a weak solution. We shall work in solenoidal function spaces and thus only consider the variables velocity and director.

Let us start with a reformulation of the Ericksen stress tensor.
For $\f v\in \V $ and $\f d \in \Hc$, we find
\begin{align}
  \nabla (( \f v \cdot \nabla ) \f d)= \nabla \f d \nabla \f v + ( \f v \cdot \nabla ) \nabla \f d \, .\label{gradF}
\end{align}
For sufficiently smooth functions $\f h: \ov{\Omega}\ra \R^3$, $\f S: \ov{\Omega}\ra \R^{3\times 3}$, we obtain
\begin{align*}
(\f v\cdot \nabla) F ( \f h, \f S)
& = \sum_{i=1}^3 \f v_i \left( \sum _{j=1}^3 \pat{F}{\f h_j } ( \f h, \f S) \pat{\f h_j }{\f x_i } + \sum_{j,k=1}^3 \pat{F}{\f S_{jk}}( \f h, \f S)\pat{ \f S_{jk}}{\f x_i} \right)\\
&= \pat{F}{\f h}(\f h, \f S) \cdot ( \f v\cdot \nabla ) \f h + \pat{F}{\f S}: ( \f v\cdot \nabla) \f S\, .
\end{align*}
With \eqref{Erik}, \eqref{Edef}, \eqref{gradF},  and integration by parts, we obtain for all $\f v\in \V$
\begin{align}
 \left(\f T^E :\nabla \f v  \right)-\left \langle \nabla \f d ^T \f q ,  \f v\right \rangle  & = \left( \nabla\f  d ^T \pat{F}{\f S}:\nabla \f v  \right) +\left\langle  \nabla\f  d ^T \di \pat{F}{\f S}, \f v  \right\rangle -
 \left( \nabla\f  d ^T \pat{F}{\f h}, \f v  \right)\notag\\
  &= \left(  \pat{F}{\f S}:\nabla\f  d  \nabla \f v  \right) +\left\langle  \di \pat{F}{\f S},(\nabla\f  d ) \f v  \right\rangle - \left(  \pat{F}{\f h},(\nabla\f  d ) \f v  \right)\notag \\
    &= \left(  \pat{F}{\f S}:\nabla\f  d  \nabla \f v  \right) +\left(  \di \pat{F}{\f S},(\f v\cdot \nabla)\f  d   \right)- \left(  \pat{F}{\f h},(\f v\cdot \nabla)\f  d    \right)\notag \\
&= \left(  \pat{F}{\f S}:\nabla\f  d  \nabla \f v  \right) -\left(  \pat{F}{\f S}:\nabla \left( (\f v\cdot \nabla)\f  d  \right)  \right)- \left(  \pat{F}{\f h},(\f v\cdot \nabla)\f  d    \right)\notag \\
 &= -\left (   \pat{ F}{\f S}: (\f v\cdot \nabla )\nabla \f d  \right )- \left (\pat{F}{\f h} ,  (\f v\cdot \nabla)\f  d \right )\notag \\
& =- \left (\nabla F , \f v \right ) = \intet{ F\, ( \di  \f v)} = 0 \, ,
\label{identi}
\end{align}
where we omitted the argument $(\f d, \nabla \f d)$ for readability.

The above identity allows us to reformulate equation~\eqref{eq:strong} by incorporating  $F$ in a redefinition of the pressure,
$p:=p+F$, and by finally replacing $\di \f T^E $ by $ -\nabla \f d^T \f q$.

\begin{definition}[Weak solution]\label{defi:weak}
Let  $( \f v_0, \f d_0)\in \Ha \times \Hb$ and $\f g \in L^2(0,T;\Vd)$.
A pair $(\f v , \f d )$ is said to be a solution to~\eqref{eq:strong}
with \eqref{abkuerzungen}, \eqref{anfang} if
\begin{align}
\begin{split}
\f v &\in L^\infty(0,T;\Ha)\cap  L^2(0,T;\V)
\cap W^{1,2}(0,T; ( \f H^2 \cap \V)^*),
\\ \f d& \in L^\infty(0,T;\Hb)\cap  L^2(0,T;\Hc) \cap W^{1,4/3}  (0,T;  \f  L^2 ),
\end{split}\label{weakreg}
\end{align}
if
\begin{subequations}\label{weak}
\begin{align}
\begin{split}
-\int_0^T (\f v(t), \partial_t\f \varphi(t)) \de t &+ \int_0^T ((\f v(t)\cdot \nabla) \f v(t), \f \varphi(t)) \de t  - \intte{\left (\nabla
\f d(t)^T \f q(t)  ,  \f \varphi(t)
\right )}\\&+ \intte{(\f T^L(t): \nabla \f \varphi(t) ) } =\intte{ \left \langle \f g (t),\f \varphi(t)\right \rangle } ,\quad
\end{split}\label{eq:velo}\\\begin{split}
-\intte{( \f d(t), \partial_t\f \psi(t) ) } &+ \intte{((\f v(t)\cdot \nabla ) \f d(t), \f \psi(t))} - \intte{\left( (\nabla \f v(t))_{\skw} \f d(t) , \f \psi(t)\right )}\\&
+\lambda\intte{\left( (\nabla \f v(t))_{\sym} \f d(t) , \f \psi(t)\right)}
+ \gamma\intte{\left(\f q(t) , \f \psi(t)\right)}=0
\end{split}
\label{eq:dir}
\end{align}%
\end{subequations}
for all solenoidal $ \f \varphi \in \mathcal{C}_c^\infty( \Omega \times (0,T);\R^3)$ and all $ \f \psi \in \mathcal{C}_c^\infty( \Omega \times (0,T);\R^3)$, and if the initial conditions are satisfied.
\end{definition}

\begin{remark} Let $(\f v , \f d )$ be a solution in the sense of Definition~\ref{defi:weak}. Then $\f d \in L^{\infty}(0,T; \Hb)$ as well as
\begin{align*}
\f d \in W^{1,4/3}(0,T;\f L^2)\subset\AC([0,T]; \f L^2) \subset\C_w([0,T]; \f L^2) \, .
\end{align*}
Moreover, there holds (see Lions and Magenes~\cite[Ch.~3, Lemma~8.1]{magenes})
\begin{align*}
\C_{w}([0,T];\f L^2)\cap L^{\infty}(0,T; \Hb) = \C_w ([0,T]; \Hb)
\end{align*}
such that $\f d \in \C_w ([0,T]; \Hb)$. Analogously, we find $\f v \in \C_w ([0,T]; \Ha)$.

The initial conditions for the Navier--Stokes-like equation~\eqref{eq:velo} and for the director equation~\eqref{eq:dir} are thus attained in the following sense:
\begin{align}
\f v(t) \rightharpoonup \f v_0 \quad \text{in } \Ha \, ,
\quad
\f d(t) \rightharpoonup \f d_0
\quad \text{in } \Hb
\quad \text{as } t \ra 0\, .\label{iniweak}
\end{align}
\end{remark}

The above notion of a weak solution can be justified as follows.

\begin{proposition}\label{propprop}
Under the regularity assumptions \eqref{weakreg}, all terms appearing in~\eqref{weak} are well-defined.
\end{proposition}

\begin{proof}
Obviously, we only have to concentrate on the nonlinear terms.
Let us start with equation~\eqref{eq:dir}.
With H\"older's inequality, we easily find that
$(\f v\cdot \nabla ) \f d \in L^{2}(0,T; \f L^{3/2})$ since
\begin{equation*}
\| (\f v\cdot \nabla ) \f d\|_{L^{2}(\f L^{3/2})}
\le \| \f v \|_{L^{\infty}(\f L^{2})} \|\nabla \f d\|_{L^{2}(\f L^6)}
\end{equation*}
and since $\f H^2 \hookrightarrow \f W^{1,6}$. Similarly, we have
\begin{equation*}
\| \sk v\f d \|_{L^{2}(\f L^{3/2})} + \| \sy v\f d \|_{L^{2}(\f L^{3/2})}
 \le c \|\nabla \f v \|_{L^2(\f L^{2})} \|\f d\|_{L^{\infty}(\f L^6)} \, ,
\end{equation*}
which shows that also $\sk v\f d , \,  \sy v\f d\in L^{2}(0,T;\f L^{3/2})$.
Note that $\f L^{3/2} \hookrightarrow \Hd$.

The term $\f q$~(see~\eqref{Edef}) is in $L^2(0,T;\Le)$. Indeed, with~\eqref{QQQQ} we can estimate
\begin{align*}
\| \f q\|_{L^2(\Le)} &\leq \left \| \pat{F}{\f h} ( \f d , \nabla \f d)  \right \|_{L^2(\Le)} + \left \| \pat{^2F}{\f S \partial\f h} ( \f d , \nabla \f d):\nabla \f d^T  \right \|_{L^2(\Le)} + \| \di ( \f \Lambda : \nabla \f d) \| _{L^2(\Le)} \\
&= I_1+I_2 +I_3 \, .
\end{align*}
Regarding the term $I_1$, we see with growth condition~\eqref{FcoercE} and Lemma~\ref{lem:nir} that
\begin{align*}
I_1 & \leq C_{\f h}\left \|  | \nabla \f d| ^{\gamma_1/2} + | \f d|^{\gamma_2/2} +1 \right \| _{L^2(L^2)} \\
& = C_{\f h} \left (  \| \nabla \f d\|_{L^{\gamma_1}(\f L^{\gamma_1})}^{\gamma_1/2} + \| \f d \|_{L^{\gamma_2}(\f L^{\gamma_2})}^{\gamma_2/2}+ T^{1/2}|\Omega|^{1/2}                \right )\\
& \leq
 c \left (  \| \nabla \f d\|_{L^{10/3}(\f L^{10/3})}^{5/3} + \| \f d \|_{L^{10}(\f L^{10})}^{5}+          1 \right )\\
& \leq c \left ( \| \f d \|_{L^2(\Hc)}\|\f d \|_{L^{\infty}(\Hb)}^{2/3}+ \| \f d \|_{L^2(\Hc)}\|\f d \|_{L^{\infty}(\Hb)}^{4}+1 \right )\, .
\end{align*}
The term $I_2$ is dealt with in a similar fashion. The growth condition~\eqref{FSh} gives
\begin{align*}
I_2 \leq C_{\f S\f h} \big \| | \nabla \f d |( | \nabla \f d | ^ { \gamma_1/2-1} + | \f d|^{\gamma_3}+1)\big \|_{L^2(L^2)}\, .
\end{align*}
An application of Young's inequality together with definition~\eqref{beta} and Lemma~\ref{lem:nir} provides that
\begin{align*}
I_2& \leq c \left  \| | \nabla \f d|^{\gamma_1/2}+ | \f d |^{\gamma_2/2} + | \nabla \f d| \right \|_{L^2(\Le)} \\
& \le c\left ( \| \nabla \f d\|_{L^{\gamma_1}(\f L^{\gamma_1})}^{\gamma_1/2} + \| \f d \|_{L^{\gamma_2}(\f L^{\gamma_2})}^{\gamma_2/2} + \| \nabla \f d\|_{L^2(\Le)}  \right )\\
& \leq c\left ( \| \nabla \f d\|_{L^{10/3}(\f L^{10/3})}^{5/3} + \| \f d \|_{L^{10}(\f L^{10})}^{5} + 1  \right )\\
& \leq c \left ( \| \f d \|_{L^2(\Hc)}\|\f d \|_{L^{\infty}(\Hb)}^{2/3} +\| \f d \|_{L^2(\Hc)}\|\f d \|_{L^{\infty}(\Hb)}^4 +1 \right )\, .
\end{align*}
Finally, the term $I_3$ can be estimated by
\begin{align*}
I_3 \leq c \| \f d \|_{L^2(\Hc)}\,
\end{align*}
since $\f \Lambda$ is a constant tensor (see~\eqref{LambdaS}).

Let us turn to the Navier--Stokes-like equation~\eqref{eq:velo}. The convection term can be shown to be in $ L^{10/7}(0,T; \f L^{15/13})$ since
 \begin{align*}
 \| ( \f v\cdot \nabla ) \f v \|_{L^{10/7}(\f L^{15/13})} \leq \| \f v \| _{ L^5 ( \f L^{30/11})} \| \f v\|_{L^2 ( \Hb)}\, ,
 \end{align*}
which can easily be shown with H\"older's inequality.
Moreover, $\f v \in  L^5 ( 0,T;\f L^{30/11})$ in view of  Lemma~\ref{cor:vreg}.
Since $\f q\in L^2(0,T; \Le)$, we easily find with H\"{o}lder's inequality that
\begin{align*}
\|\nabla\f d ^T \f q\|_{L^{5/4}(\f L^{5/4})} \le
\| \nabla \f d\|_{L^{10/3}(\f L^{10/3})}\| \f q\|_{L^2(\Le)}
\, ,
\end{align*}
where the norm of $\nabla \f d$ can be estimated as before.

It remains to estimate the Leslie stress. From definition~\eqref{Leslie} and H\"older's inequality, it follows that
\begin{align}
\begin{split}
\| \f T^L\|_{L^{10/9}(\f L^{10/9}) }& \leq c \big ( \| (\f d \cdot \sy v \f d) \f d \otimes \f d \|_ {L^{10/9}(\f L^{10/9}) }
+ \|\nabla \f v \|_ {L^{10/9}(\f L^{10/9}) }
\\ & \qquad + \| \f d \otimes \sy v \f d \| _{L^{10/9}(\f L^{10/9}) }
+ \| \f d \otimes \f e \|_{L^{10/9}(\f L^{10/9})}
\big ) \\
& \leq c \big (\| \f v \|_{L^2( \Hb)} \| \f d\|_{L^{10}(\f L^{10})}^4
+ \| \f v \|_{L^2( \Hb)}
\\&\qquad
+ \| \f  v \|_{L^2( \Hb)} \| \f d\| _{L^{10}(\f L^{10})}^2 + \| \f e \| _{L^{5/4}(\f L^{5/4})} \| \f d\| _{L^{10}( L^{10})} \big)\, .
\end{split}\label{Tl}%
\end{align}
We can further estimate the norm of $\f e$ (see~\eqref{e2}) by
\begin{align*}
\| \f e\|_{L^{5/4}(\f L^{5/4})}& \leq \| \t \f d \|_{L^{5/4}(\f L^{5/4})}+ \|( \f v \cdot \nabla ) \f d\|_{L^{5/4}(\f L^{5/4})} + \| \sk v \f d\|_{L^{5/4}(\f L^{5/4})} \\
& \leq c \| \t \f d \|_{L^{4/3}(\f L^2)}+ c \| \f v\|_{L^\infty ( \Le)} \| \f d\|_{L^2(\f W^{1, 6})} +  c \| \f v\|_{L^2(\Hb)} \| \f d\|_{L^\infty(\f L^6)}\, .
\end{align*}
All this shows that $\f T^L\in {L^{10/9}(0,T;\f L^{10/9})}$.
\end{proof}
Our main result is

\begin{theorem}[Existence of weak solutions]\label{thm:main}
Let $\Omega$ be a bounded domain of class $\C^{2}$, assume \eqref{con}, and let the free energy potential $F$ fulfil the assumptions~\eqref{regF}, \eqref{Lambda1}, \eqref{coerc1}, and \eqref{freeenergy}. For given initial data $ \f v_0 \in \Ha $, $ \f d_0\in  \Hb $ (such that $\f d_1 = 0$)  and right-hand side $ \f g\in L^2 ( 0,T; (\V)^*)$, there exists a weak solution to the Ericksen--Leslie system~\eqref{eq:strong}  with
\eqref{abkuerzungen}, \eqref{anfang}
in the sense of Definition~\ref{defi:weak}.
\end{theorem}

Let us note that one may also handle a more general right-hand side $\f g= \f g_1 + \f g_2$ with $\f g_1 \in L^1( 0, T; (\Ha)^*)$ and $ \f g_2 \in L^2 ( 0,T; (\V)^*)$
(see Tartar~\cite[Chapter 20]{tartarnav} or Simon~\cite{simon}).
For the sake of simplicity, however, we neglect $\f g_1$.

\begin{remark}\label{rem:bound}
A non-homogeneous boundary condition for $\f d$, i.e., $\f d_1 \neq 0$, can be dealt with by a standard transformation as follows.
Let $ \f d_0 \in \He(\Omega)$, $\f d_1 \in \f H ^{3/2}( \partial \Omega)$ and assume that the trace of $\f d_0$ equals $\f d_1$. As $ \Omega$ is of class $\C^{2}$, there exists a linear continuous extension operator $ \mathcal{S}: \f H^{3/2}(\partial \Omega) \ra \f H^2 ( \Omega)$ that is the right-inverse of the trace operator (see  Wloka \cite[Thm.~8.8]{wlokaeng}).
Therefore, we can replace $\f d$ in \eqref{eq:strong} by $\f d -  \mathcal{S} \f d_1$ that satisfies a homogeneous boundary condition.
The regularity of $\f d_1$ and thus of $\mathcal{S} \f d_1$ is sufficient to estimate all the terms appearing in addition.
\end{remark}

\section{Galerkin approximation and proof of the main result}\label{Sec:Proof}
In this section, we prove the main result (Theorem~\ref{thm:main}) via convergence of a Galerkin approximation. The proof is divided into the following steps:
We first (Section~\ref{sec:dis}) introduce the Galerkin scheme and deduce local-in-time existence of a solution to the approximate problem. We then (Section~\ref{sec:energy}) derive a priori estimates and conclude that solutions to the approximate problem exist globally in time. The crucial part is dealt with in Section~\ref{sec:conv}, where we
use the a priori estimates to extract a weakly convergent subsequence of the sequence of approximate solutions. We also prove weak convergence of a subsequence of the sequence of time derivatives of the approximate solutions. This implies strong convergence in a suitable norm and allows us to identify the initial values. Strong convergence is a prerequisite to handle the nonlinear variational derivative of the free energy.  Finally, we can identify the weak limits as a solution to the director  and the Navier--Stokes-like equation.

\subsection{Galerkin approximation}\label{sec:dis}
For the approximation of the Navier--Stokes-like equation, we follow
Temam~\cite[p.~27f.]{temam} and use a Galerkin basis consisting of  eigenfunctions $\f w_1, \, \f w_2 , \, \ldots \in \f H^2\cap \V $ of the Stokes operator (with homogeneous Dirichlet boundary condition). As is well known, the eigenfunctions form an orthogonal basis in $\Ha$ as well as in $\V$ and in $\f H^2\cap \V $. Let $W_n= \spa \left \{ \f w_1, \dots , \f w _n\right \} $ ($n\in \N$)
and let $P_n : \Ha \longrightarrow W_n$ denote the
$\Ha$-orthogonal projection onto $W_n$. The restriction of $P_n$ on $\V$ and $\f H^2\cap \V $ is nothing than the $\V$- and $(\f H^2\cap \V)$-orthogonal projection onto $W_n$, respectively, such that
\begin{equation*}
\|P_n\|_{\mathcal{L}(\Ha)} = \|P_n\|_{\mathcal{L}(\V)} = \|P_n\|_{\mathcal{L}(\f H^2\cap \V)} = 1\, .
\end{equation*}
Remark that $\Omega$ is of class $\C^2$.
Hence, there exists  $c>0$ such that for all $n\in \N$ and $\f v \in \f H^2 \cap \V$
\begin{equation*}
\|P_n \f v\|_{\f H^2} \le c \|\f v\|_{\f H^2} \, ,
\end{equation*}
see, e.g.,  M\'{a}lek et al.~\cite[Appendix, Thm.~4.11 and Lemma~4.26]{malek} together with Boyer and Fabrie~\cite[Prop.~III.3.17]{boyer}.

For the approximation of the director equation, we use a Galerkin basis consisting of eigenfunctions $\f z_1 , \, \f z_2 , \, \ldots$
of the differential operator corresponding to the boundary value problem
\begin{align}\label{boundaryvalueproblem}
\begin{split}
- \di \left( \f \Lambda : \nabla \f z \right) &= \f h \quad\text{in } \Omega \, , \\
\f z &= 0\quad\text{on } \partial\Omega \,.
\end{split}
\end{align}
In view of the assumptions (\ref{Lambda1}) on $\f \Lambda$, the above problem  is a symmetric strongly elliptic system that possesses a unique weak solution $\f z \in \f H^1_0$ for any $\f h \in \f H^{-1}$ (see, e.g., Chipot~\cite[Thm.~13.3]{chipot}). Its solution operator is thus a compact operator in $\f L^2$. Hence there exists an orthogonal basis of eigenfunctions $\f z_1 , \, \f z_2 , \, \ldots $ in $\Le$.

\begin{remark}\label{rem:Lambda1}
 If $\f \Lambda $ depends on $\f x$, the claims of the above paragraph are not true any more. The boundary-value problem~\eqref{boundaryvalueproblem} would  not be well-posed under this generalised condition.
Consider a tensor $\f \Lambda \in \C^{0,1}(\ov \Omega; \R^{3\times 3 \times 3 \times 3})$,
and let the strong Legendre--Hadamard condition~\eqref{Lambdaellip}  as well as the symmetry condition~\eqref{Lambdasym} be fulfilled uniformly in $\f x$.
Under this conditions, the well-posedness can be achieved by considering the boundary-value problem shifted by a multiplicative of the identity,
\begin{align}
\begin{split}
- \di \left( \f \Lambda : \nabla \f z \right) + \zeta \f z  &= \f h \quad\text{in } \Omega \, , \\
\f z &= 0\quad\text{on } \partial\Omega \,.
\end{split}\label{bvp2}
\end{align}
Here $\zeta$ is a possibly large constant.
This boundary-value problem is for a sufficiently large constant $\zeta$ well-posed (see, e.g., Chipot~\cite[Prop.~13.1 and Prop.~13.2]{chipot}) and the solution operator is again compact. The existence of eigenfunctions to this system follows by the same arguments as above.

Consequently, the Galerkin space has to be adapted for an $\f x$-dependent tensor $\f \Lambda(\f x)$
and is the span of eigenfunctions of the solution operator to the boundary-value problem~\eqref{bvp2}.
However, \eqref{identi} as well as the variational derivative have to be adapted and in particular additional terms with the derivative of $\f \Lambda $ with respect to $\f x$ occur.
It is, therefore, open whether the main result also applies to this generalisation.

\end{remark}

Moreover, the problem~\eqref{boundaryvalueproblem} is $H^2$-regular (see, e.g., Morrey~\cite[Thm.~6.5.6]{morrey} and recall that $\Omega$ is of class $\mathcal{C}^{2}$), i.e., for any $\f h \in \f L^2$ the solution $\f z$ is in $\f H^2 \cap \f H^1_0$ and there exists a constant $c_{\f \Lambda}>0$ such that
\begin{equation}\label{H2-Lambda}
\|\f z\|_{\f H^2} \le c_{\f \Lambda} \, \|\di \left( \f \Lambda : \nabla \f z \right)\|_{\f L^2}
\end{equation}
for any $\f z \in \f H^2 \cap \f H^1_0$.
This also shows that $\|\di \left( \f \Lambda : \nabla \cdot \right)\|_{\f L^2}$, $\|\Delta \cdot\|_{\f L^2}$, $\|\cdot\|_{\f H^2}$ are equivalent norms on $\f H^2 \cap \f H^1_0$ and that the eigenfunctions are  in
$\f H^2 \cap \f H^1_0$.

Again the eigenfunctions form an orthogonal basis in $\Le$. Let $ Z_n : = \spa \left \{  \f z_1, \dots , \f z_n \right \} $ ($n\in \N$) and assume $\|\f z_i\|_{\f L^2} = 1$ for $i=1, 2, \dots$. Then
\begin{equation*}
R_n : (\Hc)^* \longrightarrow Z_n \, , \quad
R_n  \f f:= \sum_{i=1}^n \langle \f f, \f z_i \rangle \f z_i
\end{equation*}
is well-defined and its restriction to $\Le$ is the
$\Le$-orthogonal projection onto $Z_n$ such that
\begin{equation*}
\|R_n\|_{\mathcal{L}(\Le)} = 1 \, .
\end{equation*}
If we equip $\Hb$ and $\Hc$ with the inner product
\begin{equation}\label{innerinner}
\left(\left( \f \Lambda : \nabla \f v  \right): \nabla \f w \right)
\text{ and }
\left(\nabla \cdot \left( \f \Lambda : \nabla \f v \right) , \nabla \cdot
\left( \f \Lambda : \nabla \f w \right) \right),
\end{equation}
respectively, then the induced norms are equivalent to the standard norms
(see also Chipot~\cite[Prop.~13.1]{chipot})
and the restriction of $R_n$ to $\Hb$ and $\Hc$ is the $\Hb$- and $\Hc$-orthogonal projection onto $Z_n$, respectively. The corresponding operator norms then equal $1$ and, with respect to the standard norms, we thus find that there is a constant $c>0$ such that for all $n\in \N$
\begin{equation*}
 \|R_n\|_{\mathcal{L}(\Hb)} \le c \, , \quad \|R_n\|_{\mathcal{L}(\Hc)} \le c\, .
\end{equation*}

Let $n\in \N$ be fixed. As usual, we consider the ansatz
\begin{align*}
\f v_n ( t)  = \sum_{i=1}^n v_n^i(t)\f w_i, \quad \f d_n(t) = \sum_{i=1}^nd_n^i(t)\f z_i \, .
\end{align*}
Our approximation reads as
\begin{subequations}\label{eq:dis}
\begin{align}
\begin{split}
( \partial_t {\f v_n}, \f w  ) +( ( \f v_n \cdot \nabla ) \f v _n, \f w  ) -(\nabla \f d_n ^ T  \f q_n , \f w  )+ \left (\f T^L_n: \nabla \f w \right )&= \left \langle  \f g   , \f w\right  \rangle,\\
\f v_n(0) &= P_n \f v_0 \,,
\end{split}
\label{vdis}\\
\begin{split}
( \partial_t \f d_n , \f z) + ( ( \f v_n \cdot \nabla ) \f d_n , \f z )
- (\skn{v} \f d _n , \f z ) + \lambda ( \syn v \f d_n , \f z)
+ \gamma (\f q_n, \f z ) &=0 \,,
\\ \f d_n(0)&= R_n\f d_0\,,
\end{split}\label{ddis}
\end{align}
for all $ \f w \in W_n$ and $ \f z \in Z_n$,
where
\begin{align}
\f q_n : = R_n \left ( \frac{\delta \mathcal{F}}{\delta \f d }( \f d_n )\right ) = R_n \left (\pat{F}{\f h}( \f d_n ,\nabla\f d _n ) - \di\pat{F}{\f S}( \f d_n , \nabla \f d_n)\right )\, ,\label{qn}
\end{align}
and
\begin{align}
\begin{split}
\f T^L_n :={}&  \mu_1 (\f d_n \cdot  \syn v \f d_n )\f d_n \otimes \f d_n+\mu_4 \syn v  - \gamma(\mu_2+\mu_3) \left (\f d_n \otimes \f q_n    \right )_{\sym} \\&
- \left (\f d_n \otimes \f q_n  \right)_{\skw} + (\mu_5+\mu_6 -\lambda(\mu_2+\mu_3))  \left ( \f d_n \otimes \syn v \f d_n \right)_{\sym}\, .
\end{split}\label{lesliedis}
\end{align}%
\end{subequations}%
The approximation $\f T^L_n$ of the Leslie stress tensor $\f T^L$ (see~\eqref{Leslie}) relies upon replacing the term $\f e$, in which unfortunately the time derivative of $\f d$ occurs, by
$\f e = -\lambda \sy v \f d - \gamma \f q$,
which follows from \eqref{dir} and \eqref{e2}. Formally, this leads with \eqref{constanten} to
\begin{align}\label{leslieleslie}
\begin{split}
\f T^L =&{}  \mu_1 (\f d\cdot \sy  v \f d )\f d \otimes \f d +\mu_4 \sy v
- \gamma(\mu_2+\mu_3)\left (\f d \otimes \f q    \right )_{\sym}
\\&
- \left (\f d \otimes \f q  \right )_{\skw} + (\mu_5+\mu_6 -\lambda(\mu_2+\mu_3))  \left ( \f d\otimes  \sy v\f d\right )_{\sym} .
\end{split}
\end{align}

It is standard to prove existence locally in time of solutions to the approximate problem in the sense of Carath\'{e}odory, i.e., of solutions that are absolutely continuous with respect to time (see, e.g.,
Hale~\cite[Chapter~I, Thm.~5.2]{hale}). Of course, the existence interval may depend on $n$. Global-in-time existence on $[0,T]$, however, follows later directly from suitable a priori estimates.

\subsection{Energy inequality and a priori estimates}\label{sec:energy}
In what follows, we derive an energy inequality and appropriate a priori  estimates for the approximate solutions.

\begin{proposition}
\label{lem:1}
Let the assumptions of Theorem~\ref{thm:main} be fulfilled and
let $\{(\f v_n , \f d_n)\} $ be a  solution to \eqref{eq:dis}. Then there holds for almost all $t$
\begin{align} \begin{split}
&\frac{\mathrm{d}}{\mathrm{d} t}  \left( \frac{1}{2}\|\f v_n\|_{\Le}^2 +  \F(\f d_n) \right) +  \mu_1\left\|\f d_n\cdot \syn v \f d_n\right\|_{L^2}^2 + \mu_4 \|\syn v \|_{\Le}^2 \\
&\quad +   (\mu_5+\mu_6- \lambda (\mu_2+\mu_3))\|\syn{v}\f d_n\|_{\Le}^2
+ \gamma \|\f q_n\|_{\Le}^2 \\
&= \langle \f g, \f v_n\rangle + \left ( \gamma (\mu_2+\mu_3) -\lambda \right ) (\f q_n , \syn v \f d_n )\, .
\end{split}\label{entro1}
\end{align}
\end{proposition}
\begin{remark}\label{rem:pa}
If Parodi's relation~\eqref{parodi} is assumed to hold for the constants appearing (see also~\eqref{constanten}) then the constant in front of the last term on the right-hand side of~\eqref{entro1} vanishes.
\end{remark}
\begin{proof}[Proof of Proposition~\ref{lem:1}]
We recall that the approximate solution $\{(\f v_n, \f d_n)\}$ is absolutely continuous in time. In order to derive the estimate asserted, we test \eqref{vdis} with $\f v_n$, \eqref{ddis} with $\f q_n$, and add both equations. This leads to
\begin{equation*}
\begin{split}
&\frac{1}{2}\br{} \| \f v_n \|_{\Le}^2  + ( \t \f d_n , \f q_n) - \langle \nabla \f d_n^T \f q_n , \f v_n \rangle  + ( ( \f v_n \cdot \nabla )\f d_n, \f q_n)
\\
&\quad + ( \f T_n^L: \nabla \f v_n)
- (\skn{v} \f d_n , \f q_n)    + \gamma \| \f q_n\|^2_{\Le}
\\
& =  \langle \f g , \f v_n\rangle
- \lambda ( \syn v \f d_n , \f q_n)
\, .
\end{split}
\end{equation*}
Here we have employed that the convection term vanishes since $\f v_n$ is solenoidal. Moreover, the projection $R_n$ maps $\Le$ into $Z_n$, which  ensures that $\f q_n$ takes values in $Z_n$.

A straightforward calculation shows that
\begin{align}
\begin{split}
\br{} \mathcal{F}(\f d_n) &= \br{} \intet{ F(\f d_n, \nabla \f d_n)  } =\intet{ \partial_t F( \f d_n, \nabla \f d_n) }\\
&=  \inte{ \pat{F}{\f h}( \f d_n, \nabla \f d_n) \cdot \t \f d_n  + \pat{F}{\f S}( \f d_n, \nabla \f d_n) : \t \nabla \f d_n } \\
&=\intet{ \t \f d_n\cdot \left( \pat{F}{\f h} ( \f d_n, \nabla \f d_n) - \di \  \pat{F}{\f S}( \f d_n, \nabla \f d_n)\right)   }\\
& = \intet{ \t \f d_n \cdot R_n\left( \pat{F}{\f h}( \f d_n, \nabla \f d_n)  - \di \  \pat{F}{\f S}( \f d_n, \nabla \f d_n)\right)    } = ( \t \f d_n , \f q_n)\, .
\end{split}
\label{Fd3}
\end{align}
In the last but one step, we used that $R_n$ is the $\Le$-orthogonal projection onto $Z_n$.

For the term with the Leslie stress tensor, we find with rules recapitulated in Section~\ref{sec:not} that
\begin{align*}
\f T^L_n : \nabla\f v_n ={}&  \mu_1 ( \f d_n \cdot  \syn{v}\f d_n)^2  +\mu_4 | \syn v|^2
\\
&+  (\mu_5+\mu_6-\lambda(\mu_2+\mu_3))| \syn{v} \f d_n |^2
 \\
&-\gamma( \mu_2+\mu_3)\, \f q_n \cdot \syn v\f d _n +  \f q_n\cdot  \skn v \f d_n \, .
\end{align*}
The assertion now follows from putting all together.
\end{proof}

An essential step in our analysis is the following energy inequality that is an adaptation of Lin and Liu~\cite[Lemma~1]{linliu3} to general free energy functions considered here. In order to ensure the dissipative character of the system, the constants
appearing are supposed to fulfil the constraints~\eqref{con}.

\begin{corollary}\label{cor1}
Under the assumptions of Theorem~\ref{thm:main} there exist $\alpha,\beta, c>0$ such that for any solution $\{(\f v_n , \f d_n)\} $ to \eqref{eq:dis} the energy inequality
\begin{align}
\begin{split}
&\frac{\text{\emph{d}}}{\text{\emph{d}} t}
\left( \frac{1}{2}\| \f v_n \|_{\Le}^2
+ \mathcal{F}(\f d_n) \right)
 + \mu_1\left \Vert \f d_n\cdot  \syn v \f d_n\right \Vert_{L^2}^2 \\
&\quad + \frac{\mu_4}{2} \|\syn v \|_{\Le}^2 + \alpha \|\syn{v}\f d_n\|_{\Le}^2
+ \beta  \|\f q_n\|_{\Le}^2  \leq c\| \f g\|_{\Vd}^2
\end{split}\label{entroin}
\end{align}
holds for almost all $t$.
\end{corollary}

\begin{proof}
The assertion is an immediate consequence of Proposition~\ref{lem:1}.

We first observe with Korn's first inequality~(see, e.g., McLean~\cite[Thm.~10.1]{mclean}) and Young's inequality that
\begin{align*}
\langle \f g, \f v_n \rangle &\leq \| \f g\|_{\Vd} \| \f v_n \|_{\V} \leq c\| \f g\|_{\Vd}  \|\syn v\|_{\Le}
\\
& \leq \frac{c^2}{2\mu_4} \| \f g\|_{\Vd}  + \frac{\mu_4	}{2}\|\syn v\|_{\Le}\, .
\end{align*}

In a second step, we find with (\ref{con}) that there is $\delta \in (0,1)$ such that
\begin{align*}
( \gamma( \mu_2+\mu_3) -\lambda )^2 \le
4 \delta^2 \gamma (\mu_5+\mu_6 - \lambda(\mu_2+\mu_3))
\end{align*}
and thus
\begin{align*}
&( \gamma(\mu_2+\mu_3)  -\lambda  ) \left (\f q_n, \syn v \f d_n \right )
 \\
&\leq
\left|  \gamma(\mu_2+\mu_3)-\lambda \right|
\| \f q_n\|_{\Le} \| \syn{v} \f d_n \|_{\Le}
\\
&\le
2 \delta \sqrt{\gamma (\mu_5+\mu_6 - \lambda(\mu_2+\mu_3))}
 \| \f q_n\|_{\Le} \| \syn{v} \f d_n\|_{\Le}
 \\
 &\le
 \delta \gamma  \| \f q_n\|_{\Le}^2 +
\delta (\mu_5+\mu_6 - \lambda(\mu_2+\mu_3)) \| \syn{v} \f d_n\|_{\Le}^2
\, .
\end{align*}
Taking $\alpha, \beta$ appropriately proves the assertion.
\end{proof}

\begin{corollary}[A priori estimates I]
\label{cor2}
Under the assumptions of Theorem~\ref{thm:main}  there holds
\begin{align}
\begin{split}
&\frac{1}{2}\| \f v _n \|_{L^\infty(\f L^2)}^2 +  \sup_{t}
\mathcal{F}(\f d_n(t))
 + \mu_1\left \Vert \f d_n \cdot \syn v \f d_n \right \Vert_{L^2(L^2)}^2
\\
&\quad + \frac{\mu_4}{2} \|\syn v \|_{L^2(\Le)}^2+ \alpha \|\syn v\f d_n\|_{L^2(\Le)}^2
+ \beta  \|\f q_n\|_{L^2(\Le)}^2
\\
& \leq  c \left (\|\f v_0\|_{\Le}^2
+  \| \f g\|_{L^2(\Vd)}^2+ \| \f d_0 \|^6_{\Hb}+1\right ).
\end{split}
\label{entrodiss}
\end{align}
on the time interval of existence.
\end{corollary}

\begin{proof}
Integrating \eqref{entroin} with respect to time implies
\begin{align*}
\begin{split}
&\frac{1}{2}\| \f v _n \|_{L^\infty(\f L^2)}^2 +  \sup_{t}
\mathcal{F}(\f d_n(t))
 + \mu_1\left \Vert \f d_n \cdot \syn v \f d_n \right \Vert_{L^2(L^2)}^2  \\ &\quad + \frac{\mu_4}{2} \|\syn v \|_{L^2(\Le)}^2+ \alpha \|\syn v\f d_n\|_{L^2(\Le)}^2
+ \beta  \|\f q_n\|_{L^2(\Le)}^2
\\
&\leq c\left ( \|\f v_0\|_{\Le}^2 +
\mathcal{F}(R_n \f d_0) + \| \f g\|_{L^2(\Vd)}^2\right )
\end{split}
\end{align*}
since $\f v_n(0) = P_n \f v_0$ and $\f d_n(0) = R_n \f d_0$.
Thus it remains to estimate $\F( R_n \f d_0)$ independently of~$n$.

In view of the smoothness of the free energy potential $F$, we find with the fundamental theorem of calculus
\begin{align*}
\begin{split}
F( R_n\f d_0 , \nabla R_n\f d_0) = F( \f d_0, \nabla \f d_0) &+ \int_0^1 \pat{F}{\f h}(\f d_0^n(s), \nabla\f d_0^n(s) ) \de s \cdot ( R_n\f d_0- \f d_0) \\& + \int_0^1\pat{F}{\f S}(\f d_0^n(s), \nabla\f d_0^n(s) ) \de s  :\nabla( R_n\f d_0- \f d_0)\, ,
\end{split}
\end{align*}
where $\f d_0^n(s) := s R_n\f d_0 + (1-s) \f d_0$ ($s \in [0,1]$).
Applying the fundamental theorem of calculus again to the second term gives
\begin{align*}
\begin{split}
F( R_n\f d_0 , \nabla R_n\f d_0) ={}& F( \f d_0, \nabla \f d_0) + \int_0^1 \pat{F}{\f h}(\f d_0^n(s), \nabla\f d_0^n(s) ) \de s \cdot ( R_n\f d_0- \f d_0)
\\&  + \int_0^1 \pat{F}{\f S}(0, 0 ) \de s  :\nabla( R_n\f d_0- \f d_0)
\\& + \int_0^1 \int_0^1  \pat{^2 F}{\f S \partial\f h}(\tau \f d_0^n(s), \tau \nabla\f d_0^n(s) )\de \tau \cdot  \f d_0^n(s)  \de s :  \nabla  (R_n \f d_0- \f d_0)
\\&
+\int_0^1 \int_0^1  \pat{^2 F}{\f S^2}(\tau \f d_0^n(s), \tau \nabla\f d_0^n(s) )\de \tau : \nabla   \f d_0^n(s)  \de s :  \nabla  (R_n \f d_0- \f d_0)\,.
\end{split}
\end{align*}

Integrating over $\Omega$
gives
\begin{align*}
\begin{split}
\F(R_n \f d_0) ={}&   \F( \f d_0 ) + \int_0^1 \left (\pat{F}{\f h}(\f d_0^n(s), \nabla\f d_0^n(s) ) ,  R_n\f d_0- \f d_0\right ) \de s
\\&  + \int_0^1 \left (\pat{F}{\f S}(0, 0 )  ;\nabla( R_n\f d_0- \f d_0)\right ) \de s
\\& + \int_0^1 \int_0^1  \left ( \pat{^2 F}{\f S \partial\f h}(\tau \f d_0^n(s), \tau \nabla\f d_0^n(s) ) \cdot  \f d_0^n(s)   ;  \nabla  (R_n \f d_0- \f d_0)\right ) \de \tau\de s
\\&
+\int_0^1 \int_0^1 \left (  \pat{^2 F}{\f S^2}(\tau \f d_0^n(s), \tau \nabla\f d_0^n(s) ) : \nabla   \f d_0^n(s)   ;  \nabla  (R_n \f d_0- \f d_0)\right )\de \tau \de s \\
={}&\F( \f d_0 )+I_1+I_2+I_3+I_4
\, ,
\end{split}
\end{align*}

such that
\begin{equation*}
\left|\F(R_n \f d_0) - \F( \f d_0 )\right| \le  |I_1| + |I_2| + |I_3|+ |I_4| \, .
\end{equation*}

Let us consider the term~$I_1$. Invoking the growth conditions~\eqref{FcoercE} and Young's inequality yields
\begin{align*}
\begin{split}
|I_1| \leq{}& \intet{ \int_0^1 \left|
\pat{F}{\f h}(\f d_0^n(s), \nabla\f d_0^n(s) )\right|  \de s \,
| R_n\f d_0- \f d_0| } \\
\leq{}& C_{\f h}  \intet{ \int_0^1 \left( | \nabla \f d_0^n(s)|^{\gamma_1/2} +| \f d_0^n(s)|^{\gamma_2/2} +1   \right) \de s \, |R_n\f d_0- \f d_0 |}\,.
\end{split}
\end{align*}
Using  Young's and H\"{o}lder's inequality as well as Sobolev's embedding theorem thus leads to
\begin{align*}
|I_1|
\leq {}&c \intet{ \int_0^1 \left( | \nabla \f d_0^n(s)|^{5/3} +| \f d_0^n(s)|^{5} +1   \right) \de s \, |R_n\f d_0- \f d_0 |}\\
\leq {}&  c\int_{\Omega }
\Big (  | \nabla R_n\f d_0|^{5/3}+ |R_n\f d_0|^5
+  | \nabla \f d_0|^{5/3} + | \f d_0|^5  + 1 \Big )|R_n\f d_0- \f d_0 |\de \f x
\\
\leq {}&c \left ( \| \nabla R_n \f d_0 \| _{\Le} ^2 + \| R_n \f d_0\|_{\f L^6}^6 + \| \nabla  \f d_0 \| _{\Le} ^2 + \|  \f d_0\|_{\f L^6}^6+1 \right )^{5/6}
\|R_n\f d_0- \f d_0 \|_{\f L^6}
\\
\leq{}& c \left ( \|  R_n \f d_0 \| _{\Hb} ^6 + \|  \f d_0\|_{\Hb}^6+1 \right )^{5/6}
\|R_n\f d_0- \f d_0 \|_{\Hb}
\end{align*}

Due to the continuity of the derivative of $F$ with respect to $\f S$, the term $I_2$ can be estimated by
\begin{align*}
| I_2| \leq \left |\pat{F}{\f S}(0, 0 ) \right | \| \nabla( R_n\f d_0- \f d_0)\|_{\f L^1} \leq c \| R_ n \f d_0 - \f d_0 \|_{\Hb} \,.
\end{align*}

For the term $I_3$, we get with~\eqref{FSh} and estimating $\tau$ by 1 that
\begin{align*}
|I_3| \leq{}& \intet{ \int_0^1 \left|
\pat{^2 F}{\f S\partial\f h}(\f d_0^n(s), \nabla\f d_0^n(s) )\right||  \f d_0^n(s)|  \de s \,
|\nabla( R_n\f d_0- \f d_0)| } \\
\leq{}& C_{\f S\f h}  \intet{ \int_0^1 \left( \left (| \nabla \f d_0^n(s) |^{\gamma_1/2-1} + |  \f d_0^n(s)|^{\gamma_3 }+1\right )  |  \f d_0^n(s)|   \right) \de s \, |\nabla (R_n\f d_0- \f d_0) |}\,.
\end{align*}
Since $\gamma_1 $ can be estimated by $10/3$ and $\gamma_3$ by $2$, we obtain
 \begin{align*}
|I_3|
\leq{}& c  \intet{ \int_0^1 \left (| \nabla \f d_0^n(s) |^{2/3}|  \f d_0^n(s)|  + |  \f d_0^n(s)|^{3 }+1\right )   \de s \, |\nabla (R_n\f d_0- \f d_0) |}\,.
\end{align*}
An application of Young's and H\"older's inequality gives
\begin{align*}
|I_3| \leq{}&c \intet{ \int_0^1 \left( | \nabla \f d_0^n(s)|^1  +| \f d_0^n(s)|^{3} +1   \right) \de s \, |\nabla(R_n\f d_0- \f d_0 )|}\\
 \leq{}& c \left ( \| R_n\f d _0 \| _{\Hb}^2 +\| R_n\f d _0 \| _{\f L^6}^6  + \|  \f d_0\|_{\Hb }^2 + \|  \f d_0\|_{\f L^6 }^6+1 \right )^{1/2}
\|R_n\f d_0- \f d_0 \|_{\Hb}\\
 \leq{}& c \left ( \|  R_n \f d_0 \| _{\Hb} ^6 + \|  \f d_0\|_{\Hb}^6+1 \right )^{1/2}
\|R_n\f d_0- \f d_0 \|_{\Hb}
\,.
\end{align*}

For the term $I_4$, we observe with~\eqref{LambdaS} that
\begin{equation*}
|I_4| \le c \| \f d_0^n(s) \|_{\Hb}\| R_n\f d_0  - \f d_0 \|_{\Hb}
\le c \left( \| R_n\f d_0 \|_{\Hb}+ \| \f d_0 \|_{\Hb}\right) \| R_n\f d_0  - \f d_0 \|_{\Hb}\, .
\end{equation*}

Finally, we come up with
\begin{align*}
\left|\F(R_n \f d_0) - \F(\f d_0)\right|  \le c \left( \| R_n \f d_0\|_{\Hb}^5 + \| \f d_0\|_{\Hb}^5 +1\right)\| R_n\f d_0  - \f d_0 \|_{\Hb}\, .
\end{align*}
With similar estimates as before, one can show that
\begin{equation*}
|\F(\f d_0)| \le c \left(\| \f d_0\|_{\Hb}^6 +1 \right)
\end{equation*}
if $\f d_0 \in \Hb$.
Due to the boundedness of the projection $R_n$ as an operator in $\Hb$, we see that $\F(R_n \f d_0)$ is bounded (independently of $n$) as long as $\f d_0 \in \Hb $.
This proves the assertion.
\end{proof}
\begin{remark}
The proof of Corollary~\ref{cor2} also shows that $\F(R_n \f d_0)$ converges to $\F ( \f d_0)$. The proof of Corollary~\ref{cor2} is done in such a way that it is still valid for the modified assumptions of Section~\ref{sec:more}.
\end{remark}

Due to the strong coercivity assumptions on the free energy potential $F$ and its derivatives, we are able to deduce a priori estimates in spaces with rather strong norms.

\begin{corollary}[A priori estimates II]
\label{cor:apri2}
Under the assumptions of Theorem~\ref{thm:main}  there is a constant $C>0$ such that for all $n\in\N$
\begin{align}
\begin{split}
&\| \f v _n \|_{L^\infty(\f L^2)}^2 +  \| \f d_n\|_{L^{\infty}(\Hb)}^2   + \left \Vert \f d_n \cdot  \syn v \f d_n \right \Vert_{L^2(L^2)}^2 + \|\syn v \|_{L^2(\Le)}^2 \\
&+ \|\syn v\f d_n\|_{L^2(\Le)}^2
+  \| \Delta \f d_n\|_{L^2(\Le)}^2
\leq C
\, .
\end{split}
\label{apri2}
\end{align}
\end{corollary}

\begin{proof}
The coercivity condition~\eqref{coerc1} implies
\begin{align}
 \mathcal{F}(\f d_n(t)) =  \intet{F( \f d_n(t), \nabla \f d_n(t)) }& \geq \eta_1 \|\nabla \f d_n(t)\|_{\Le}^2 -\eta_2\| \f d_n(t)\|_{\Le}^2 - \eta_3|\Omega| \,
 \label{coercivity}
\end{align}
The director equation~\eqref{ddis} tested with $\f d_n$ gives
\begin{align*}
( \t \f d_n , \f d_n ) + ( (\f v_n \cdot \nabla)\f d_n , \f d_n  ) - ( (\nabla\f v_n)_{\skw} \f d_n , \f d_n ) + \lambda ( (\nabla \f v_n)_{\sym} \f d_n , \f d_n ) + \gamma ( \f q_n, \f d_n) =0\,.
\end{align*}
Integration in time and using $ 2(\t \f d_n , \f d_n)=  \t \| \f d_n\|_{\Le}^2 $, the vanishing divergence of $\f v_n$ and the skew-symmetry of $\skn v$ shows that
\begin{align*}
 \frac{1}{2}\| \f d_n (t)\|_{\Le}^2 +  \inttet{\left (\lambda ( (\f v_n)_{\sym} \f d_n , \f d_n ) + \gamma ( \f q_n, \f d_n)\right )}= \frac{1}{2}\| \f d_0\|_{\Le}^2 \,.
\end{align*}
The norm of the director can thus be estimated using H\"older's and Young's inequality and in the second step Poincar\'e's inequality
\begin{align*}
\| \f d_n (t)\|_{\Le}^2 \leq{}& \inttet{\left (\frac{\alpha
}{2\eta_2} \| (\nabla \f v_n)_{\sym} \f d_n \|_{\Le}^2 + \frac{\beta}{2
\eta_2}\| \f q_n \|_{\Le}^2  + c \| \f d_n\|_{\Le}^2
\right )}   + \| \f d_0\|_{\Le}^2 \\ \leq{}& \inttet{\left (\frac{\alpha
}{2\eta_2} \| (\nabla \f v_n)_{\sym} \f d_n \|_{\Le}^2 + \frac{\beta}{2
\eta_2}\| \f q_n \|_{\Le}^2  + c \|\nabla  \f d_n\|_{\Le}^2
\right )}  + \| \f d_0\|_{\Hb}^2  \,.
\end{align*}
Applying this estimate together with~\eqref{coercivity} and~\eqref{entroin} gives
\begin{align*}
\frac{1}{2}\| \f v _n (t)\|_{\f L^2}^2 +  \eta_1 \| \nabla \f d_n(t)\|_{\Le}^2   &+ \mu_1\inttet{\left \Vert \f d_n \cdot  \syn v \f d_n \right \Vert_{L^2}^2}  \\ &+ \frac{1}{2} \inttet{\left (\mu_4 \|\syn v \|_{\Le}^2+ \alpha \|\syn v\f d_n\|_{\Le}^2
+ \beta  \|\f q_n\|_{\Le}^2\right )}  \\ \leq{}& \frac{1}{2}\|\f v_0\|_{\Le}^2
+c \left ( \| \f g\|_{L^2(\Vd)}^2+ \| \f d_0 \|^6_{\Hb}+1\right )\\&+ \eta_2 c \inttet{\| \nabla \f d_n\|_{\Le}^2 }+\eta_2 \| \f d_0\|_{\Hb}^2  +\eta_3 |\Omega|\,.
\end{align*}
The lemma of Gronwall and taking the supremum over all  $t\in[0,T]$ shows the estimate
\begin{align}
\begin{split}
\frac{1}{2}\| \f v _n \|_{L^\infty(\f L^2)}^2 +  \eta_1 \| \nabla \f d_n\|_{L^{\infty}(\Le)}^2   &+ \mu_1\left \Vert \f d_n \cdot  \syn v \f d_n \right \Vert_{L^2(L^2)}^2  \\ &+ \frac{\mu_4}{2} \|\syn v \|_{L^2(\Le)}^2+ \frac{\alpha}{2} \|\syn v\f d_n\|_{L^2(\Le)}^2
+ \frac{\beta}{2}  \|\f q_n\|_{L^2(\Le)}^2  \\ \leq{}& c\left (\|\f v_0\|_{\Le}^2 + \mathcal{F}(\f d_0)
+  \| \f g\|_{L^2(\Vd)}^2+ \| \f d_0 \|^6_{\f H^{1}} + 1\right )e^{cT} \\ =:{}& K\,
\end{split}\label{apri1}
\end{align}
and thus, the boundedness of the director in the $L^\infty(0,T; \He)$ norm.


Moreover, \eqref{qn}, \eqref{QQQQ}, and \eqref{abcabc} leads to
\begin{align*}
|\f q_n|^2\geq \frac{1}{2} \left |R_n \di \left(  \f \Lambda :  \nabla \f d_n\right) \right  |^2 - 4 \left |R_n \pat{^2 F}{\f S\partial \f h}(\f d_n, \nabla \f d_n) : \nabla \f d_n^T \right |^2 - 4 \left|R_n \pat{F}{\f h} (\f d_n, \nabla \f d_n)\right |^2 \, .
\end{align*}
The orthogonality properties of $R_n$ then imply
\begin{align*}
\|\f q_n\|_{\Le}^2\geq \frac{1}{2} \left\|\di \left(  \f \Lambda :  \nabla \f d_n\right) \right\|_{\Le}^2 - 4 \left\|\pat{^2 F}{\f S\partial \f h}(\f d_n, \nabla \f d_n) : \nabla \f d_n^T \right\|_{\Le}^2 - 4 \left\|\pat{F}{\f h} (\f d_n, \nabla \f d_n)\right\|_{\Le}^2 \, .
\end{align*}
The estimates (\ref{H2-Lambda}), \eqref{FSh}, and~\eqref{FcoercE} show that
there are constants $c_1, c > 0$ such that
for all $t$
\begin{align}
\|\f q_n\|_{\Le}^2 \geq & c_1 \| \Delta \f d_n\|_{\Le}^2 -
c \left ( \|\nabla \f d_n\|_{\f L^{\gamma_1 }}^{\gamma_1}
+ \left \lVert | \f d_n |^{\gamma_3}
| \nabla \f d_n |\right \rVert _{\Le}^2
+ \|\nabla \f d_n\|_{\f L^2}^2
+ \| \f d_n \|_{\f L^{\gamma_2} }^{\gamma_2}+1 \right) .\label{qnabsch}
\end{align}
With (\ref{beta}), an application of Young's inequality (recall that $\gamma_3 = 0$ if $\gamma_1 = 2$),
\begin{align*}
\left \lVert | \f d_n |^{\gamma_3} | \nabla \f d_n |\right \rVert _{\Le}^2
& \leq  \frac{2}{\gamma_1 } \| \nabla \f d_n \|_{\f L^{\gamma_1}}^{\gamma_1} + \frac{\gamma_1-2}{\gamma_1} \| \f d_n \| _{\f L^{2\gamma_3  \gamma_1/(\gamma_1-2)}}^{2\gamma_3\gamma_1/(\gamma_1-2)} \\
& =  \frac{2}{\gamma_1 } \| \nabla \f d_n \|_{\f L^{\gamma_1} }^{\gamma_1}
+ \frac{\gamma_1-2}{\gamma_1} \| \f d_n \|_{\f L^{\gamma_2}}^{\gamma_2} \, ,
\end{align*}
and
\begin{equation*}
\|\nabla \f d_n\|_{\f L^2}^2 \le c \|\nabla \f d_n\|_{\f L^{\gamma_1}}^2
\le 1 + c \|\nabla \f d_n\|_{\f L^{\gamma_1}}^{\gamma_1} \, ,
\end{equation*}
we come up with
\begin{align*}
\|\f q_n\|_{\Le}^2 \geq  c_1 \| \Delta \f d_n\|_{\Le}^2 -
c \left ( \|\nabla \f d_n\|_{\f L^{\gamma_1 }}^{\gamma_1}
+ \| \f d_n \|_{\f L^{\gamma_2}}^{\gamma_2}
+ 1 \right)
\end{align*}
and thus with
\begin{align*}
\|\f q_n\|_{L^2(\Le)}^2 \geq  c_1 \| \Delta \f d_n\|_{L^2(\Le)}^2 -
c \left ( \|\nabla \f d_n\|_{L^{\gamma_1}(\f L^{\gamma_1})}^{\gamma_1}
+ \| \f d_n \|_{L^{\gamma_2}(\f L^{\gamma_2})}^{\gamma_2}
+ 1 \right) .
\end{align*}
Lemma~\ref{lem:nir} yields
\begin{align*}
\|\nabla \f d_n\|_{L^{\gamma_1}(\f L^{\gamma_1})}^{\gamma_1}
\le c \, \|\Delta \f d_n\|_{L^2(\f L^2)}^{\theta_1}
\|\f d_n\|_{L^\infty (\Hb)}^{\gamma_1 - \theta_1} \, , \;
\theta_1 = \frac{3}{2} \, (\gamma_1 - 2)
 \, ,
\\
\|\f d_n\|_{L^{\gamma_2}(\f L^{\gamma_2})}^{\gamma_2}
\le c \, \|\Delta \f d_n\|_{L^2(\f L^2)}^{\theta_2}
\|\f d_n\|_{L^\infty (\Hb)}^{\gamma_2 - \theta_2} \, , \;
\theta_2 = \frac{1}{2} \, (\gamma_2 - 6)
\, .
\end{align*}
Young's inequality now leads to
\begin{align*}
\|\f q_n\|_{L^2(\Le)}^2 \geq  \frac{c_1}{2}\, \| \Delta \f d_n\|_{L^2(\Le)}^2 - c \left ( \|\f d_n\|_{L^\infty (\Hb)}^p
+ 1 \right)
\end{align*}
with
\begin{equation*}
p= 2\max\left( \frac{\gamma_1-\theta_1}{2-\theta_1},
\frac{\gamma_2-\theta_2}{2-\theta_2}\right)
= 2\max\left( \frac{6-\gamma_1}{10-3\gamma_1},
\frac{6 + \gamma_2}{10 - \gamma_2}\right) < \infty
.
\end{equation*}
Because of \eqref{apri1}, we already know that $\eta_1 \|\f d_n\|_{L^\infty (\Hb)}^2$ is bounded by $K$. Hence, \eqref{entrodiss} leads to
\begin{align*}
&\frac{1}{2}\| \f v _n \|_{L^\infty(\f L^2)}^2 +  \eta_1 \| \f d_n\|_{L^{\infty}(\Hb)}^2   + \mu_1\left \Vert \f d_n \cdot  \syn v \f d_n \right \Vert_{L^2(L^2)}^2 + \frac{\mu_4}{2} \|\syn v \|_{L^2(\Le)}^2 \\
&+ \frac{\alpha}{2} \|\syn v\f d_n\|_{L^2(\Le)}^2
+ \frac{c_1\beta}{4}\, \| \Delta \f d_n\|_{L^2(\Le)}^2
\leq K
+ c \beta \left ( \left(\frac{K}{\eta_1}\right)^{\frac{p}{2}} + 1 \right)
\, ,
\end{align*}
which proves the assertion.
\end{proof}

The above a priori estimates ensure that the approximate solutions exist on $[0,T]$ (see, again, Hale~\cite[Chapter~I, Thm.~5.2]{hale}).

We are now going to estimate the time derivative of $\f d_n$ and $\f v_n$ in appropriate norms.

\begin{proposition}
Under the assumptions of Theorem~\ref{thm:main}  there is a constant $C>0$ such that for all $n\in\N$
\begin{align}
 \| \partial_t \f v_n\|_{L^{2}((\f H^2 \cap \V)^*)} \le C \, .\label{vnt}
\end{align}
\end{proposition}

\begin{proof}
Recall that $P_n$ is the $(\f H^2 \cap \V)$-orthogonal projection onto $W_n$. We thus find with \eqref{vdis} for all $t\in [0,T]$ and all
$ \f\varphi \in \f H^2 \cap \V$
\begin{align*}
|\langle \t\f v_n , \f\varphi  \rangle|
={}& |(\t\f v_n , P_n\f\varphi)|
=  \left|
\langle \f g , P_n\f\varphi \rangle
-
\left(
 ( \f v_n \cdot \nabla ) \f v_n , P_n\f\varphi\right)
+\left(  \nabla \f d_n^T \f q _n , P_n\f\varphi\right)
- \left( \f T^L_n : \nabla P_n \f\varphi \right)
\right|
\\
\le{}& \|\f g\|_{\Vd} \|P_n \f\varphi \|_{\V}
+
\|( \f v_n \cdot \nabla ) \f v_n\|_{\f L^{1}}\|P_n \f \varphi \|_{\f L^{\infty}}
\\
&+ \|\nabla \f d_n^T \f q _n \|_{\f L^{1}}\|P_n \f \varphi \|_{\f L^{\infty}}
+ \| \f T^L_n \|_{\f L^{6/5}} \|\nabla P_n \f \varphi \|_{\f L^{6}} \, .
\end{align*}
Since $\f H^2$ is continuously embedded in $\f H^1$, $\f L^{\infty}$,
and $\f W^{1,6}$, we find
\begin{align*}
\|\t\f v_n\|_{(\f H^2 \cap \V)^*} \le c \left(
\|\f g\|_{\Vd} + \|( \f v_n \cdot \nabla ) \f v_n\|_{\f L^{1}}
+ \|\nabla \f d_n^T \f q _n \|_{\f L^{1}}
+ \| \f T^L_n \|_{\f L^{6/5}}  \right)
\end{align*}
and thus
\begin{align*}
\|\t\f v_n\|_{L^{2}((\f H^2 \cap \V)^*)} \le{}& c \left(
\|\f g\|_{L^{2}(\Vd)}
+ \|( \f v_n \cdot \nabla ) \f v_n\|_{L^{2}(\f L^{1})}
\right.
\\
& + \left. \|\nabla \f d_n^T \f q _n \|_{L^{2}(\f L^{1})}
+ \| \f T^L_n \|_{L^{2}(\f L^{6/5})}  \right).
\end{align*}

With H\"{o}lder's inequality and Lemma~\ref{cor:vreg}, we see that
\begin{equation*}
\|( \f v_n \cdot \nabla ) \f v_n\|_{L^{2}(\f L^{1})}
\le \|\f v_n\|_{L^\infty(\f L^{2})} \|\nabla \f v_n\|_{L^2(\f L^{2})}
\quad \text{and}
\quad\left \lVert \nabla \f d_n ^T\f q_n\right \rVert _{ L^{2}(\f L^{1})} \leq  \left \lVert \nabla \f d_n \right \rVert _{ L^\infty(\f L^ {2})} \left \lVert \f q_n \right \rVert _{ L^{ 2}(\Le )}
\, .
\end{equation*}
In view of \eqref{apri2}, \eqref{apri1}, and Korn's inequality, the terms on the right-hand sides of the foregoing estimates are bounded.

Finally, we observe with \eqref{lesliedis} the estimate
\begin{align*}
 \| \f T^L_n \|_{L^{2}(\f L^{6/5})}
 \le{}& c \left( \|(\f d_n\cdot  \syn v \f d_n )\f d_n \otimes \f d_n \|_{L^{{2}}( \f L^{{6}/{5}})}
 + \|\nabla \f v_n\|_{L^{2}( \f L^{{6}/{5}})}
\right.
\\
& +\left. \|\f d_n \otimes \f q_n\|_{L^{2}( \f L^{{6}/{5}})}
 +  \|  \f d_n \otimes  \syn v \f d_n\|_{L^{2}( \f L^{{6}/{5}})}
 \right),
\end{align*}
where (again with H\"{o}lder's inequality)
\begin{align*}
\|(\f d_n\cdot  \syn v \f d_n )\f d_n \otimes \f d_n \|_{L^{2}( \f L^{{6}/{5}})} &\leq \| \f d_n \cdot  \syn v \f d_n \|_{L^2( L^2)} \| \f d_n\|_{L^{\infty}(\f L^{6})}^2 \, ,\\
\|\f d _n \otimes \f q_n \|_{L^{{2}}( \f L^{{6/5}})} &\leq
c \|\f d _n \otimes \f q_n \|_{L^{{2}}( \f L^{{3/2}})} \leq
c \| \f d_n\|_{L^{\infty}(\f L^{6})} \| \f q_n \|_{L^2(\Le)} \, ,
\\
\| \f d_n \otimes  \syn v \f d_n\|_{L^{{2}}( \f L^{{6/5}})}
&\leq c
\| \f d_n \otimes  \syn v \f d_n\|_{L^{{2}}( \f L^{{3/2}})}
\\
&\leq  c\| \f d_n\|_{L^{\infty}(\f L^{6})} \| \syn v \f d_n \|_{L^2(\Le)}\, .
\end{align*}
which proves the assertion because of  \eqref{apri2} and \eqref{apri1}.
\end{proof}

\begin{proposition}
Under the assumptions of Theorem~\ref{thm:main}  there is a constant $C>0$ such that for all $n\in\N$
\begin{align}
 \| \partial_t \f d_n\|_{L^{4/3}(\Le)} \le C \, .\label{dtn}
\end{align}
\end{proposition}

\begin{proof}
Recall that $R_n$ is the $\Le$-orthogonal projection onto $Z_n$. We thus find with \eqref{ddis} for all $t\in [0,T]$
\begin{align}
\nonumber
\|\partial_t\f d_n\|_{\Le} =&{}
\sup_{{\|\f\psi\|_{\Le}\leq 1}} | ( \partial_t\f d_n , \f\psi)|
= \sup_{{\|\f\psi\|_{\Le}\leq 1}} | ( \partial_t \f d_n , R_n \f\psi ) |
\\
\nonumber
\leq&{} \sup_{{\|\f\psi\|_{\Le}\leq 1}}
\left\|  -(\f v_n\cdot \nabla )\f d_n +\skn{v}\f d_n - \lambda \syn v \f d_n - \gamma \f q_n
\right\|_{\Le} \| R_n \f\psi\|_{\Le}
\\
\leq&{}
\left\|  (\f v_n\cdot \nabla )\f d_n \right\|_{\Le}
+ \left\|\skn{v}\f d_n \right\|_{\Le}
+ |\lambda| \left\|\syn v \f d_n \right\|_{\Le}
+ \gamma \left\|\f q_n
\right\|_{\Le}
 \, .
\label{dntabs}
\end{align}
In view of (\ref{apri1}), we see that
\begin{equation*}
\left\|\syn v \f d_n \right\|_{L^{4/3}(\Le)} \text{ and }
\left\|\f q_n \right\|_{L^{4/3}(\Le)}
\end{equation*}
are bounded.
It remains to consider the first two terms on the right-hand side of
\eqref{dntabs}.
With H\"older's inequality and Lemma~\ref{lem:nir}, we find
\begin{align*}
\| ( \f v_n \cdot \nabla ) \f d_n\| _{L^{4/3}(\Le)}
\le \|\f v_n\|_{L^2(\f L^6)} \| \nabla \f d_n\| _{L^{4}(\f L^3)}
\le c \|\f v_n\|_{L^2(\Hb)} \|\f d_n\| _{L^\infty(\Hb)}^{1/2}
\|\f d_n\| _{L^2(\Hc)}^{1/2}\, .
\end{align*}
Note that, employing Korn's inequality, all terms on the right-hand side are bounded in view of \eqref{apri2}. Similarly, we find that
\begin{align*}
\|\skn{v} \f d_n \| _{L^{4/3}(\Le)}
\le 2\|\f  v_n\|_{L^2( \Hb )}\| \f d_n \|_{ L^{4}( \f L^{\infty} )}
\le c \|\f  v_n\|_{L^2( \Hb )}
\|\f d_n\| _{L^\infty(\Hb)}^{1/2}
\|\f d_n\| _{L^2(\Hc)}^{1/2}
\end{align*}
is bounded.
\end{proof}

\subsection{Convergence of the approximate solutions}\label{sec:conv}

In what follows, we consider a sequence of approximate solutions as $n\to\infty$. The a priori estimates for the approximate solutions imply then the following results.

\begin{corollary}\label{lem:limits}
Under the assumptions of Theorem~\ref{thm:main}  there is a subsequence (not relabeled) of the sequence of solutions to the approximate problem~\eqref{eq:dis} such that
\begin{subequations}\label{wkonv}
\begin{align}
   \f v_{n }&\stackrel{*}{\rightharpoonup}  \f v \quad &&\text{ in } L^{\infty} (0,T;\Ha)\,,\label{w:vstern}\\
 \f v_{n }&\rightharpoonup  \f v \quad &&\text{ in }  L^{2} (0,T;\V)\,,\label{w:v}\\
\f q_n &\rightharpoonup  \ov{\f q} \quad &&\text{ in }  L^{2} (0,T;\Le)\,,\label{w:E}\\
\f d_{n }&\stackrel{*}{\rightharpoonup}  \f d \quad &&\text{ in } L^{\infty} (0,T;\Hb)\,,\label{w:dstern}\\
\f d_{n }&\stackrel{ }{\rightharpoonup}  \, \f d \quad &&\text{ in } L^{2} (0,T;\Hc)\,,\label{w:d}
\\
\partial_t \f v_n &\rightharpoonup \partial_t \f v  \quad&& \text{ in }
L^{2}( 0,T; (\f H^2 \cap \V)^*) \, , \label{ww:vtime}
\\
\partial_t \f d_n &\rightharpoonup \partial_t \f d  \quad &&\text{ in }
L^{4/3}( 0,T; \Le) \, , \label{ww:dtime}
\\
\f v_n &\to \f v  \quad&& \text{ in }
L^{p}( 0,T; \Ha) \text{ for any } p \in [1,\infty)
\, , \label{s:v}
\\
\f d_n &\to \f d  \quad&& \text{ in }
L^q(0,T;L^q) \text{ for any } q\in [1,10) ,
\label{s:dL10}
\\
\f d_n &\to \f d  \quad &&\text{ in }
L^r(0,T; \f W^{1,r}) \text{ for any } r \in [1,10/3)
\label{s:d}
\, .
\end{align}
\end{subequations}
\end{corollary}

\begin{proof}
The existence of weakly and weakly$^*$ convergent subsequences immediately follows, by standard arguments, from the a priori estimates \eqref{apri2} and \eqref{apri1} together with Korn's inequality and from
\eqref{vnt}, \eqref{dtn}  together with the definition of the weak time derivative.
The strong convergence follows from the Lions--Aubin compactness lemma
(see Lions~\cite[Th\'eor\`eme~1.5.2]{lions}).
With respect to $\f v_n$, we observe that $\V$ is compactly embedded
in $\Ha$, which implies strong convergence in
$L^2(0,T;\Ha)$ and together with the boundedness in $L^{\infty}(0,T;\Ha)$ also in $L^p(0,T;\Ha)$ for any $p\in [1,\infty)$.

With respect to $\f d_n$, we observe that $\Hc$ is compactly embedded in $\Hb$, which implies strong convergence in $L^2(0,T;\Hb)$ and together with the boundedness in $L^{\infty}(0,T;\Hb)$ also in $L^p(0,T;\Hb)$ and thus in
$L^p(0,T;\f L^6)$
for any $p\in [1,\infty)$.
 Moreover, $\f d_n$ is also bounded in $L^{10}(0,T;\f L^{10})$. Indeed,
 With Lemma~\ref{lem:nir}, we find
\begin{equation*}
\| \f d_n\|_{L^{10}(\f L^{10})} \le c  \| \f d_n \|_{L^2(\Hc)}^{1/5} \| \f d_n \|_{L^\infty(\Hb)}^{4/5} \, ,
\end{equation*}
 This implies \eqref{s:dL10}.
Finally, $\f d_n$ is also bounded in $L^{10/3}(0,T;\f W^{1,10/3})$ since
\begin{equation*}
\|\f d_n\|_{L^{10/3}(\f W^{1,10/3})}
\le c \| \f d_n \|_{L^2(\Hc)}^{3/5} \| \f d_n \|_{L^\infty(\He)}^{2/5}
\end{equation*}
because of Lemma~\ref{lem:nir}. This proves \eqref{s:d}.
\end{proof}

\begin{corollary}\label{cor:initial}
Under the assumptions of Theorem~\ref{thm:main}
the limits $\f v$ and $\f d$ from Corollary~\ref{lem:limits} satisfy
\begin{equation*}
\f v(0) = \f v_0 \text{ and } \f d(0) = \f d_0 \, .
\end{equation*}
\end{corollary}

\begin{proof}
Due to the convergence results \eqref{ww:vtime} for the time derivative of the velocity field, we get
$\f v \in W^{1,2}(0,T;(\f H^2 \cap \V)^*)
\hookrightarrow \C([0,T];(\f H^2 \cap \V)^*)$ and together with \eqref{ww:vtime}  that
for any $\f w \in \f H^2 \cap \V$ and $\omega(t) = (T-t)/T$
\begin{equation*}
\langle \f v_n (0) - \f v (0) , \f w  \rangle = -
\int_0^T \left(  \langle \f v_n (t) - \f v (t) ,\f  w \omega'(t)\rangle
+ \langle \t \f v_n (t) - \t \f v (t) ,\f w \omega(t)\rangle
 \right) \de t \to 0 \, ,
\end{equation*}
which shows that \begin{equation*}
\f v_n(0) \rightharpoonup \f v(0) \text{ in } (\f H^2 \cap \V)^* .
\end{equation*}
Moreover, we know that
\begin{equation*}
\f v_n(0) = P_n \f v_0 \to \f v_0 \text{ in } \Ha \, ,
\end{equation*}
which proves $\f v(0) = \f v_0$.

Analogously, one proves that $\f d(0) = \f d_0$.
\end{proof}

With the following proposition, we identify the limit $\bar{\f q}$ in \eqref{w:E}.
\begin{proposition}\label{Eweak}
Under the assumptions of Theorem~\ref{thm:main}, the limit $\bar{\f q}$  in~\eqref{w:E} is given by $\bar{\f q} = \f q$, where $\f q$ is given by \eqref{Edef}.
\end{proposition}

\begin{proof}
In what follows, we do not relabel the subsequence that exists in view of Corollary~\ref{lem:limits}.
We already know \eqref{w:E} and wish to establish now weak convergence of $\f q_n $ to $ \f q$ in $L^{2}(0,T;\Hd)$.
Recalling that $R_n$ is the $\Le$-orthogonal projection onto $Z_n$,
we find for all $\f \psi  \in L^2( 0,T;\Hb)$ that
\begin{align}
&\intte{ \langle \f q_n(t) - \f q(t) ,  \f \psi(t) \rangle }\notag\\
= &\intte{ \left \langle \pat{F}{\f h} ( \f d_n(t), \nabla \f d_n(t)) - \di \pat{F}{\f S}     ( \f d_n(t), \nabla \f d_n(t)), R_n \f \psi(t) \right \rangle} \notag\\
 &-\intte{ \left \langle\pat{F}{\f h} ( \f d(t), \nabla \f d(t)) - \di \pat{F}{\f S}(\f d(t), \nabla \f d(t)), \f \psi(t)   \right \rangle     }\notag \\
 =
 &\intte{ \left \langle \pat{F}{\f h} ( \f d_n(t), \nabla \f d_n(t)) - \di \pat{F}{\f S}     ( \f d_n(t), \nabla \f d_n(t)), R_n \f \psi(t) - \f \psi(t) \right \rangle} \notag\\
&
 +\intte{ \left \langle \pat{F}{\f h} ( \f d_n(t), \nabla \f d_n(t)) - \pat{F}{\f h} ( \f d(t), \nabla \f d(t))  , \f \psi(t)   \right \rangle
  }\notag
\\
&- \intte{ \left \langle \di  \pat{F}{\f S}     ( \f d_n(t), \nabla \f d_n(t))- \di \pat{F}{\f S}(\f d(t), \nabla \f d(t)), \f \psi(t) \right \rangle}
 =: I_{1,n} + I_{2,n} + I_{3,n} \, .\label{qkonv}
\end{align}

Regarding the term $I_{1,n}$, we note that
\begin{align*}
&\left\|\pat{F}{\f h} ( \f d_n(t), \nabla \f d_n(t)) - \di \pat{F}{\f S}     ( \f d_n(t), \nabla \f d_n(t)) \right\|_{L^2(\f L^2)}
\\
&\le
c \left ( \| \f d_n \|_{L^2(\Hc)}\|\f d_n \|_{L^{\infty}(\Hb)}^4 + \| \f d_n \|_{L^2(\Hc)}+ 1 \right )
\, ,
\end{align*}
which can be shown as in the proof of Proposition~\ref{propprop}.
This, together with \eqref{apri2}, shows the boundedness of the term above.
Moreover, $R_n$ is the $\Le$-orthogonal projection onto $Z_n$ such that for all $\f \psi  \in L^2( 0,T;\Le)$
\begin{align*}
\lim_{n\ra \infty} \| R_n \f \psi - \f \psi \|_{L^2(\Le)}= 0\, .
\end{align*}
This shows that $I_{1,n}$ converges to $0$ as $n\to\infty$.

Let us now consider the term $I_{2,n}$. Because of \eqref{s:dL10} and \eqref{s:d}, we observe that (passing to a subsequence if necessary)
\begin{equation}
\f d_n(\f x,t)  \to \f d(\f x,t)  \, , \quad
\nabla \f d_n(\f x,t)  \to \nabla \f d(\f x,t)\label{pktw}
\end{equation}
for almost all $(\f x,t) \in \Omega \times (0,T)$. Moreover,
$|\f d_n(\f x,t)|$ is dominated by a function in
$L^q(0,T;L^q)$ ($q\in [1,10)$)
and $|\nabla\f d_n(\f x,t)|$ is dominated by a function in
$L^r(0,T; \f L^r)$ ($r \in [1,10/3)$). The growth condition \eqref{FcoercE} then shows that
\begin{equation*}
\left|\pat{F}{\f h} ( \f d_n(\f x, t), \nabla \f d_n(\f x, t))
\right| \le C_{\f h} \left( |\nabla \f d_n(\f x, t)|^{\gamma_1/2}
+ |\f d_n(\f x, t)|^{\gamma_2/2}
+1 \right)
\end{equation*}
is dominated by a function in $L^2(0,T;\Le)$.
With the continuity of $\pat{F}{\f h}$ and Lebesgue's theorem on dominated convergence, we thus find that
$I_{2,n}$ converges to $0$ as $n\to\infty$.

For the term $I_{3,n}$, we find (see \eqref{QQQQ}) that
\begin{align*}
I_{3,n} ={}&
 \intte{ \left \langle  \left(\f \Lambda : \nabla (\f d_n(t) - \f d(t)) \right) :
\nabla \f \psi(t) \right \rangle}
\\
& - \intte{ \left \langle
\pat{^2 F}{\f S\partial \f h}
( \f d_n(t), \nabla \f d_n(t)) : \nabla \f d_n(t)^T
- \pat{^2 F}{\f S\partial \f h}
(\f d(t), \nabla \f d(t))  : \nabla \f d(t)^T , \f \psi(t) \right \rangle}.
\end{align*}
The first term on the right-hand side, which is linear, converges to $0$ because of \eqref{w:d}. The second term can be dealt with similarly as $I_{2,n}$. In particular, \eqref{FSh} together with Young's inequality provides that
\begin{align*}
\left|\pat{^2 F}{\f S\partial \f h}
( \f d_n(\f x, t), \nabla \f d_n(\f x, t))
: \nabla \f d_n(\f x , t)^T \right| &\le
C_{\f S\f h} \left( | \nabla \f d_n(\f x ,t) |^{\gamma_1/2-1}+|\f d_n(\f x, t)|^{\gamma_3}  + 1 \right)
\left| \nabla \f d_n(\f x , t) \right|
\\
&\le c \left(
|\nabla \f d_n(\f x, t)|^{\gamma_1/2}
+ |\f d_n(\f x, t)|^{\gamma_2/2}
+ 1
 \right)
\end{align*}
is dominated by a function in $L^2(0,T;\Le)$.
\end{proof}

We are now ready to prove the main result.

\begin{proof}[Proof of Theorem~\ref{thm:main}]
It only remains to prove that the limit $(\f v, \f d)$
from Corollary~\ref{lem:limits} satisfies the original problem in the sense of Definition~\ref{defi:weak}. This is shown by passing to the limit in the approximate problem \eqref{eq:dis}.

Let us start with the approximation \eqref{vdis} of the Navier--Stokes-like equation. In view of Corollary~\ref{lem:limits}, we already know that the term incorporating the time derivative converges. Moreover, we find with
\eqref{s:v} convergence of the convection term such that for all solenoidal $\f\varphi \in \mathcal{C}_c^\infty( \Omega \times (0,T);\R^3)$
\begin{equation*}
\int_0^T ((\f v_n(t)\cdot \nabla) \f v_n(t), \f \varphi(t)) \de t
\to
\int_0^T ((\f v(t)\cdot \nabla) \f v(t), \f \varphi(t)) \de t \, .
\end{equation*}
With Proposition~\ref{Eweak}, \eqref{w:E}, and \eqref{s:dL10}, we find that
\begin{equation*}
\intte{\left \langle\nabla \f d_n^T(t) \f q_n(t)  ,  \f \varphi(t)
\right \rangle}
\to
\intte{\left \langle\nabla\f d^T(t) \f q(t)  ,  \f \varphi(t)
\right \rangle}\,.
\end{equation*}
With respect to the term incorporating the Leslie tensor, we only focus on the first term that is the least regular one. With \eqref{w:v} and \eqref{s:dL10}, we find that
\begin{align*}
&\intte{\left \langle (\f d_n(t) \cdot  (\nabla \f v_n(t))_{\sym}
\f d_n(t) )\f d_n(t) \otimes \f d_n(t)
: \nabla \f \varphi(t)
\right \rangle}
\\
&\to
\intte{\left \langle (\f d(t) \cdot  (\nabla \f v(t))_{\sym}
\f d(t) )\f d(t) \otimes \f d(t)
: \nabla \f \varphi(t)
\right \rangle}
\,.
\end{align*}
This, together with similar observations for the other terms,
shows that
\begin{equation*}
\intte{(\f T^L_n (t): \nabla \f \varphi(t) ) }
\to \intte{(\f T^L (t): \nabla \f \varphi(t) ) } \, ,
\end{equation*}
where $\f T^L$ is given by \eqref{leslieleslie}, which is equivalent to \eqref{Leslie}.

Regarding the approximation \eqref{ddis} of the director equation, we observe convergence of the term incorporating the time derivative because of \eqref{ww:dtime}. With \eqref{s:v} and \eqref{s:dL10}, we obtain
for all $\f\psi \in \mathcal{C}_c^\infty( \Omega \times (0,T);\R^3)$
\begin{equation*}
\int_0^T ((\f v_n(t)\cdot \nabla) \f d_n(t), \f \psi(t)) \de t
\to
\int_0^T ((\f v(t)\cdot \nabla) \f d(t), \f \psi(t)) \de t \, .
\end{equation*}
With \eqref{w:v} and \eqref{s:dL10}, we find that
\begin{equation*}
\intte{\left ( (\nabla \f v_n(t))_{\skw} \f d_n(t) , \f \psi(t)\right )}
\to
\intte{\left ( (\nabla \f v(t))_{\skw} \f d(t) , \f \psi(t)\right )}
\end{equation*}
as well as
\begin{equation*}
\intte{\left ( (\nabla \f v_n(t))_{\sym} \f d_n(t) , \f \psi(t)\right )}
\to
\intte{\left ( (\nabla \f v(t))_{\sym} \f d(t) , \f \psi(t)\right )}.
\end{equation*}
For the term with $\f q_n$, we employ \eqref{w:E} together with Proposition~\ref{Eweak}.

All this shows that the limit $(\f v, \f d)$ of the approximate solutions satisfy the original equations. Moreover, Corollary~\ref{cor:initial} shows that the initial conditions are also fulfilled.
\end{proof}
\section{Nonlinear principal part\label{sec:more}}
In this section, we study the more general case~\eqref{tildeF} and prove the main result as stated in Theorem~\ref{thm:main} but with~\eqref{LambdaS} replaced by~\eqref{Lambda2}.
\subsection{Assumptions on the second derivative\label{sec:assThe}}
To handle more general free energies with a possible nonlinear principal part, we consider now a more general model.
More specifically, the result of Theorem~\ref{thm:main} remains true, if the  assumption~\eqref{Lambda1} is replaced by
\begin{equation}\label{Lambda2}
\frac{\partial^2 F}{\partial \f S^2}(\f h ,\f S) = \f \Lambda  + \f \Theta (\f h , \f S)  \, ,
\end{equation}
where $ \f \Theta  \in \mathcal{C}(\R^3 , \R^{3\times 3} ;\R^{3\times 3\times 3\times 3})$ is sufficiently small such that $ | \f \Theta ( \f h , \f S) | \leq c_{\f \Theta}$ with  $c_{\f \Theta} = c_{\f\Lambda} /(16 c_{\f H^2}) $. Here $c_{\f \Lambda}$ denotes the constant of estimate~\eqref{H2-Lambda} and $ c_{\f H^2}$ denotes the constant of the equivalent norm estimate of $ \| \cdot \|_{\f H^2} \leq  c_{\f H^2}   \|  \cdot\|_{\Hc} =c_{\f H^2}\| \Delta \cdot \|_{\Le}  $ (see also Section~\ref{sec:not}).
The proof of Theorem~\ref{thm:main} is similar under this modified assumption. We are commenting only on the necessary changes.
\begin{theorem}
Let $\Omega$ be a bounded domain of class $\C^{2}$, assume \eqref{con}, and let the free energy potential $F$ fulfil the assumptions~\eqref{regF}, \eqref{Lambda2}, \eqref{coerc1}, and \eqref{freeenergy}. For given initial data $ \f v_0 \in \Ha $, $ \f d_0\in  \Hb $ (such that $\f d_1 = 0$)  and right-hand side $ \f g\in L^2 ( 0,T; (\V)^*)$, there exists a weak solution to the Ericksen--Leslie system~\eqref{eq:strong}  with
\eqref{abkuerzungen}, \eqref{anfang}
in the sense of Definition~\ref{defi:weak}.
\end{theorem}

\subsection{Estimate of the variational derivative}
Under assumption~\eqref{Lambda2}, the variational derivative (compare with equation~\eqref{QQQQ}) becomes
\begin{align*}
\f q
&= - \di ( \f \Lambda : \nabla \f d )  - \f \Theta (\f d, \nabla \f d)\drei
\nabla (\nabla {\f d})^T  - \pat{^2 F}{\f S\partial \f h}(\f d, \nabla \f d) : (\nabla \f d)^T + \pat{F}{\f h} (\f d, \nabla \f d)
\, .
\end{align*}
For arbitrary $\f a_i \in \R^3$, $i \in \{ 1,2,3,4\}$ one finds that
\begin{align*}
| \sum_{i=1}^4 \f a_i |^2
={}&
|\f a_1|^2 + | \f a_2|^2+ | \f a_3|^2 + | \f a _4|^2   + 2 \f a _1 \cdot (\f a_2 +  \f a_3 +\f a_4) + 2 \f a _2 \cdot (\f a_3 +  \f a _4 )+ 2\f a_3 \cdot \f a_4\notag\\
\geq{}&
\frac{1}{2}|\f a_1|^2 + | \f a_2|^2+ | \f a_3|^2 + | \f a _4|^2  - 2 | \f a_2 +  \f a_3 +\f a_4|^2 - | \f a_2|^2 - | \f a_3 + \f a_4|^2 - | \f a_3|^2 - |\f a_4|^2 \notag\\
\geq{}& \frac{1}{2}\, |\f a_1|^2 -4 | \f a_2|^2 -  10(|\f a_3|^2+|\f a_4|^2)\,
\end{align*}
and thus\begin{align*}
|\f q|^2\geq \frac{1}{2} \left | \di ( \f \Lambda:\nabla \f d) \right  |^2 &- 4 \left |\f  \Theta (\f d, \nabla \f d) \drei \nabla (\nabla \f d)^T \right |^2\\&- 10 \left | \pat{^2 F}{\f S\partial \f h}(\f d, \nabla \f d) : (\nabla \f d)^T \right |^2 - 10 \left |\pat{F}{\f h} (\f d, \nabla \f d)\right |^2 \, .
\end{align*}
Inserting this estimate into Corollary~\ref{cor:apri2} yields
\begin{align}
|\f q_n|^2\geq \frac{1}{2} \left |R_n   \di ( \f \Lambda : \nabla  \f d_n  )\right  |^2 &- 4 \left | R_n \left (\f \Theta (\f d_n , \nabla \f d_n ) \drei \nabla (\nabla \f d_n )^T \right )\right |^2\\&- 10 \left |R_n \left (\pat{^2 F}{\f S\partial \f h}(\f d_n , \nabla \f d_n ) : (\nabla \f d_n )^T \right )\right |^2 - 10 \left |R_n\pat{F}{\f h} (\f d_n , \nabla \f d_n )\right |^2 \, .\label{qgeneral}
\end{align}
Again the orthogonality property of $R_n$ and the estimate~\eqref{H2-Lambda} yields
\begin{multline}
\frac{1}{2} \left \|R_n   \di (\f  \Lambda : \nabla  \f d_n  )\right  \|_{\Le}^2 - 4 \left \| R_n \left (\f \Theta (\f d_n , \nabla \f d_n ) \drei \nabla (\nabla \f d_n )^T \right )\right \|_{\Le}^2  \geq \\ \frac{c_{\f \Lambda}}{2} \| \Delta \f d _n \|_{\Le}^2 - 4 \left \|\f  \Theta (\f d_n , \nabla \f d_n ) \drei \nabla (\nabla \f d_n )^T \right \|_{\Le}^2\,.\label{nummer}
\end{multline}
The assumptions on $\f \Theta $ (see Section~\ref{sec:assThe}) guarantee that
\begin{align*}
\left \| \f \Theta (\f d_n , \nabla \f d_n ) \drei \nabla (\nabla \f d_n )^T \right \|_{\Le} \leq \left \| \f \Theta (\f d_n , \nabla \f d_n \right )\|_{\f L^\infty } \| \f d_n \|_{\f H^2} \leq c_{\f \Theta} c_{\f H^2}  \| \Delta \f d_n \|_{\Le}\leq \frac{c_{\f \Lambda}}{16 }\| \Delta \f d_n \|_{\Le}\,.
\end{align*}
Inserting this into~\eqref{nummer} then shows that~\eqref{qgeneral} together with~\eqref{FcoercE} yields the estimate~\eqref{qnabsch}.
Now we can proceed as in the proof of Corollary~\eqref{cor:apri2}.
\subsection{Convergence of the nonlinear part}
The only other change in the proof of Theorem~\ref{thm:main} under the new assumptions~\eqref{Lambda2} is needed in the limiting procedure in the nonlinear terms in Proposition~\ref{Eweak}.
We consider the term $I_{3,n}$ of equation~\eqref{qkonv}.
For this term, we find with the fundamental theorem of calculus
that
\begin{align*}
I_{3,n} ={}&
\intte{ \left ( \pat{F}{\f S}( \f d_n (t), \nabla \f d_n (t))- \pat{F}{\f S}(\f d(t), \nabla \f d(t)) , \nabla \f \psi(t) \right )
}
\\
={}& \intte{\left ( \int_0^1 \pat{F}{\f S\partial \f h}( \f d^n_s(t), \nabla \f d^n_s(t)) \cdot ( \f d_n (t) -\f d(t))  \de s
: \nabla \f \psi (t) \right )
}\\
{}&+
\intte{\left ( \int_0^1  \f \Theta ( \f d^n_s(t), \nabla \f d^n_s(t)) : ( \nabla \f d_n (t) - \nabla \f d(t) )\de s
: \nabla \f \psi (t) \right )}
\\&+
 \intte{ \left (  \left (\f \Lambda :( \nabla \f d_n(t) - \nabla \f d(t))\right )  :
\nabla \f \psi(t) \right )}= J_{1,n}+J_{2,n}+J_{3,n}
\,,
\end{align*}
where $\f d^n_s(t) : = s \f d_n (t) + ( 1-s)\f d(t)$.
For the term $J_{1,n}$, the growth condition~\eqref{FSh}, Young's inequality, and~\eqref{beta} show that
\begin{align*}
&\left|\pat{^2 F }{\f S\partial \f h}
( \f d^n_s(\f x, t), \nabla \f d^n_s(\f x, t))
\cdot ( \f d_n(\f x , t)-\f d(\f x, t) ) \right| \\&\le
C_{\f S\f h} \left(| \nabla \f d^n_s(\f x ,t ) |^{\left ({\gamma_1}/{2}-1\right
)}+ |\f d^n_s(\f x, t)|^{\gamma_3}  + 1 \right)
\left (
\left|  \f d_n(\f x , t) \right| + \left|  \f d(\f x , t) \right|  \right )
\\
&\le c \left(
|\nabla \f d^n_s(\f x ,t )|^{{\gamma_1}/{2}}
+ |\f d^n_s(\f x ,t )|^{{\gamma_2}/{2}}
+
\left|  \f d_n(\f x , t) \right|^{{\gamma_1}/{2}}  + \left|  \f d(\f x , t) \right|^{{\gamma_1}/{2}} + 1
 \right)  \\
&\le c \left(
|\nabla \f d_n (\f x ,t )|^{{\gamma_1}/{2}}
+ |\f d_n (\f x ,t )|^{{\gamma_2}/{2}}+|\nabla \f d(\f x ,t )|^{{\gamma_1}/{2}}
+ |\f d(\f x ,t )|^{{\gamma_2}/{2}}
+ 1
 \right)\,.
\end{align*}
Thus, the above function is dominated by a function in $L^2(0,T; \Le)$ (see also~\eqref{pktw} and note that $\gamma_2 <10$). The continuity of $\pat{F}{\f S \partial \f h}$ and the point wise convergence~\eqref{pktw} imply  $| J_{1,n} | \ra 0  $ as $n \to \infty $ in view of Lebesgue's theorem on dominated convergence.

For the terms $J_{2,n}$ and $J_{3,n}$, we observe that those terms can be estimated due to the assumptions on~$\f \Theta$ by
\begin{align*}
|J_{2,n}+J_{3,n}| \leq {}& ( c_{\f \Theta } + | \f \Lambda|) \| \nabla \f d_n  - \nabla \f d \|_{L^2(\Le)} \| \nabla \f \psi\|_{L^2(\Le)}\,.
\end{align*}
The strong convergence~\eqref{s:d} with $r=2$ shows that  $|J_{2,n}+J_{3,n}|\ra 0$, for $n \ra \infty$.

\section{Some examples}\label{sec:ex}

\subsection{Dirichlet energy with Ginzburg--Landau penalisation}
System \eqref{eq:strong}--\eqref{Edef} with the assumptions on the free energy~\eqref{regF} and~\eqref{freeenergy} is especially a generalisation of the models considered by Lin and Liu~\cite{linliu3} and by Cavaterra et al.~\cite{allgemein}.
Therefore, the free energy considered in \cite{linliu3,allgemein} fits into our setting, i.e., the free energy function~\eqref{penal}
with $k_1,\varepsilon>0$ fulfils the hypothesis on the free energy potential in Section~\ref{free} with $\f \Lambda$ being a multiple of the identity mapping $\R^{3\times 3}$ into $\R^{3\times 3}$ and $\gamma_1=2$, $\gamma_2 = 6$.

\subsection{Electromagnetic field effects}
If the influence of an electromagnetic field is taken into account, which is essential since the desirable anisotropic effects are controlled in such a way,  the  function~\eqref{FH}
is needed (see de Gennes~\cite{gennes}).
If the magnetic field is bounded such that $|\f H(\f x ,t)| \leq C$ for all $(\f x ,t)\in \Omega\times [0,T]$, the energy~\eqref{FH} fits into our model. Especially, it fulfils the assumptions~\eqref{coerc1} and~\eqref{FcoercE}.

As the assumption~\eqref{FcoercE} suggests, functions $\tilde{F}$ (see~\eqref{tildeF}) which only depend on the director $\f d$ and not on its gradient are incorporated in our setting as long as they are continuously differentiable, bounded from below  and of polynomial growth of a degree strictly less than $6$.

\subsection{Additional degrees of freedom}
Leslie recognised in \cite{leslie} that the system~\eqref{ELeq} possesses two additional degrees of freedom.
He proposed to alter the system~\eqref{eq:strong} by adding $\f d \otimes \f b$ to the derivative of the free energy potential with respect to the gradient of the director, and by adding $ \bar{b}\f  d - \nabla \f d\f b$ to the derivative of the free energy with respect to the director with certain constant vector $\f b \in \R^d$ and constant scalar $\bar{b}\in \R$.

These two changes can be introduced into the system by replacing the free energy $F$ by a new function $F_A$ defined by~\eqref{FA}.

With this function, we find that
\begin{align*}
\pat{F_A}{\f S} ( \f d, \nabla \f d) &= \pat{F}{\f S}( \f d, \nabla \f d) - \f d \otimes \f b
\end{align*}
and
\begin{align*}
\pat{F_A}{\f h}(\f d, \nabla \f d) & = \pat{F}{\f h} ( \f d, \nabla \f d)- \nabla \f d \, \f b + \bar{b}\f d \, ,
\end{align*}
as proposed by Leslie (see~\cite{leslie}).
 What remains is to check  that the system~\eqref{eq:strong} is not altered somewhere else. The only other term depending on the free energy is the Ericksen stress. To calculate the altered Ericksen stress, tested with a solenoidal, sufficiently smooth function $\f w$, we use the calculation~\eqref{identi} and obtain
\begin{align*}
\langle \f T^E_A: \nabla \f w \rangle & = \langle \nabla \f d^T \f q_A  ,  \f w \rangle    = \left \langle \nabla \f d^T\left ( \pat{F_A}{\f h}(\f d, \nabla \f d) - \di \pat{F_A}{\f S} ( \f d, \nabla \f d)\right ), \f w \right \rangle \\
& =\left \langle \nabla \f d^T \left ( \pat{F}{\f h} ( \f d, \nabla \f d)- \nabla \f d \,\f b + \bar{b}\f d -\di  \left ( \pat{F}{\f S}( \f d, \nabla \f d) - \f d \otimes \f b \right )    \right ), \f w \right \rangle \\
& = \left \langle \nabla \f d^T \left ( \pat{F}{\f h} ( \f d, \nabla \f d)  -\di   \pat{F}{\f S}( \f d, \nabla \f d) + \bar{b}\f d    \right ), \f w  \right \rangle \\
& = \langle \nabla \f d^T \f q , \f w \rangle + \frac{\bar{b}}{2}  ( \nabla | \f d|^2, \f w )  = \langle \nabla \f d^T \f q , \f w \rangle = \langle \f T^E : \nabla \f w\rangle\, .
\end{align*}
The last but one equation holds since $\f w$ is a solenoidal function.
We thus see that the system is only changed by the new free energy potential~\eqref{FA}  as proposed by Leslie.
The model with these additional degrees of freedom thus fits into our framework.

\subsection{Simplified Oseen--Frank energy}
The Oseen--Frank energy~\eqref{oseen}  fits only  into our setting for a particular choice of the constants appearing. Since
\begin{align*}
 (\f d \cdot  (\curl \f d) )^2 + | \f d \times( \curl \f d) |^2 = | \f d |^2  | \curl \f d |^2
\end{align*}
and since in the classical Ericksen--Leslie model the norm of the  director is supposed to be one, it is convenient to consider the following reformulation of the Oseen--Frank energy for the case $k_2=k_3$:
\begin{align*}
F_{\{k_2=k_3\}}(\f d , \nabla \f d ) = k_1 ( \di \f d )^2 + k_2 | \curl \f d |^2 + \alpha( \tr(\nabla \f d^2)- (\di \f d)^2)\, .
\end{align*}

For $k_1,k_2 >0$, this energy fulfils the assumptions of Section~\ref{free}: A simple but tedious calculation shows that
\begin{align}\label{LambdaDef}
\begin{split}
|\nabla \f d|^2 &= \nabla \f d : \f \Lambda^0 : \nabla \f d
\quad\text{with } \f\Lambda^0_{ijkl} = \delta_{ik}\delta_{jl}
\\
(\di \f d)^2 & =
\nabla \f d : \f \Lambda^1 : \nabla \f d \quad \text{with } \f \Lambda^1_{ijkl} = \delta_{ij}\delta_{kl}
\\
\tr (\nabla \f d ^2) &= \nabla \f d : \f \Lambda^2 : \nabla \f d \quad \text{with } \f \Lambda^2_{ijkl} = \delta_{il}\delta_{jk}
\\
|\nabla \times \f d|^2 &= \nabla \f d : \f \Lambda^3 : \nabla \f d \quad \text{with }  \f \Lambda^3 = \f \Lambda^0 - \f \Lambda^2 \, ,
\end{split}
\end{align}
where $\f \Lambda^0$ is the identity mapping $\R^{3\times 3}$ into $\R^{3\times 3}$. It follows that
\begin{equation*}
F_{\{k_2=k_3\}}(\f d , \nabla \f d ) =
\frac{1}{2}\, \nabla \f d : \f \Lambda : \nabla \f d
\quad \text{with }
\f \Lambda = 2 k_2 \f \Lambda^0 + 2 (k_1- \alpha) \f \Lambda^1 - 2(k_2 -\alpha) \f \Lambda^2 \, .
\end{equation*}
For $\f a, \, \f b \in \R^3$, we find
\begin{equation*}
\f a\otimes \f b : \f \Lambda : \f a\otimes \f b
=
2 k_2 |\f a|^2 |\f b|^2 + 2 (k_1 - k_2) (\f a \cdot \f b)^2 \, ,
\end{equation*}
and the strong Legendre--Hadamard condition is fulfilled for any $\alpha \in \R$ and any $k_1 , k_2 > 0$.

\subsection{Scaled {Oseen--Frank} energy}
It does not seem to be possible to include the general Oseen--Frank energy in the presented setting, but to include its anisotropic character. We consider an energy, where the non-quadratic terms are scaled appropriately.
The energy is given by~\eqref{scaledO} with sufficiently small constants $k_3$ and $k_4$.

The associated tensor of fourth order  $\f \Lambda$ is given by
\begin{align*}
\f \Lambda = k_1 \f \Lambda^1 + \min\{k_2,k_3\} \left ( \f \Lambda^0 - \f \Lambda ^2 \right )\,,
\end{align*}
where $\f \Lambda^i$ for $i=0,1,2$ is defined in~\eqref{LambdaDef}.
Let $c_{\f \Lambda} $ be the associated coercivity constant of the estimate~\eqref{H2-Lambda}.
A careful calculation and estimate of the partial derivatives of the free energy potential~\eqref{scaledO}, shows that they fulfill the assumptions of Section~\ref{free}
and Section~\ref{sec:more} for $ s>\frac{1}{6}$, and $k_3+k_4$ sufficiently small.
The condition on $s$ follows from the property~\eqref{FSh} and the smallness condition from property~\eqref{Lambda2} and the estimate on $\f \Theta$.
 If in addition $ s \leq \frac{1}{4} $, it can be seen by roughly estimating the second partial derivative of $F$ with respect to $\f S$ that the constants $k_3$ and $k_4$ have to be chosen such that   $$53(k_3+k_4)\leq \frac{c_{\f \Lambda}}{c_{\f H^2}}\,.$$

\addcontentsline{toc}{section}{References}
\bibliographystyle{abbrv}

\end{document}